\documentclass[11pt,reqno]{amsart}
\usepackage[utf8]{inputenc}
\usepackage{amssymb}
\usepackage[all]{xy}
\usepackage{xcolor}
\usepackage{hyperref}
\newtheorem{thm}{Theorem}[section]
\newtheorem{lem}[thm]{Lemma}
\newtheorem{prop}[thm]{Proposition}
\newtheorem{cor}[thm]{Corollary}

\theoremstyle{definition}
\newtheorem{dfn}[thm]{Definition}

\theoremstyle{remark}
\newtheorem{remark}[thm]{Remark}

\newcommand{\CA}{{\mathcal{A}}}
\newcommand{\CC}{{\mathcal{C}}}
\newcommand{\CD}{{\mathcal{D}}}

\newcommand{\CF}{{\mathcal{F}}}
\newcommand{\CG}{{\mathcal{G}}}

\newcommand{\CN}{{\mathcal{N}}}
\newcommand{\CI}{{\mathcal{I}}}

\newcommand{\CL}{{\mathcal{L}}}

\newcommand{\CZ}{{\mathcal{Z}}}
\newcommand{\CB}{{\mathcal{B}}}

\newcommand{\CP}{{\mathcal{P}}}
\newcommand{\CW}{{\mathcal{W}}}

\newcommand{\af}{\alpha}
\newcommand{\bt}{\beta}
\newcommand{\gm}{\gamma}
\newcommand{\dt}{\delta}

\newcommand{\ld}{\lambda}

\newcommand{\kp}{\kappa}

\newcommand{\N}{{\mathbb{N}}}

\newcommand{\WCB}{{\widehat{\mathcal{B}}}}

\makeatletter
\@namedef{subjclassname@2020}{%
  \textup{2020} Mathematics Subject Classification}
\makeatother

\begin{document}


\title[ A generalized uniqueness theorem for GBDS ]
{ A generalized uniqueness theorem for generalized Boolean dynamical systems }

\author[E. J. Kang]{Eun Ji Kang}
\address{
Research Institute of Mathematics, Seoul National University, Seoul 08826, 
Korea} \email{kkang33\-@\-snu.\-ac.\-kr }



\begin{abstract}
We characterize the canonical diagonal subalgebra of the $C^*$-algebra associated with a generalized Boolean dynamical system. We also introduce a particular commutative subalgebra, which we call the abelian core, in our $C^*$-algebra. We then establish a uniqueness theorem under the assumptions that $\CB$ and $\CL$ are countable, which says that a $*$-homomorphism of our $C^*$-algebra is injective if and only if its restriction to the abelian core is injective.
 \end{abstract}

\maketitle

\section{Introduction}

The $C^*$-algebras associated with generalized Boolean dynamical systems was introduced in \cite{CasK1} as a generalization of graph $C^*$-algebras, and their structures have been extensively studied by several authors.
Among other developments, the boundary path space $ \partial E$ (see Definition \ref{def:boundary path space}) associated with a generalized Boolean dynamical system was introduced in \cite[Definition 7.5]{CasK1}, and it is used to construct groupoid models (\cite{CasK1}, \cite{CasK2}) and a partial action (\cite{CasK2}) associated with the given generalized Boolean dynamical system.

In the study of the $C^*$-algebra $C^*(\CB, \CL, \theta, \CI_\af)$ associated with a generalized Boolean dynamical system $(\CB, \CL, \theta, \CI_\af)$, a central focus is to determine conditions that guarantee the injectivity of a given $*$-homomorphism of $C^*(\CB, \CL, \theta, \CI_\af)$. 
 This is commonly known as a ``uniqueness theorem''. 
One of the well-known uniqueness theorems in the context of generalized Boolean dynamical systems is the Cuntz-Krieger uniqueness theorem (\cite[Theorem 3.6]{CaK3}).
  It asserts that 
  the underlying Boolean dynamical system $(\CB, \CL,\theta)$ satisfies Condition (L), that is, every cycle has an exit if and only if 
 a $*$-homomorphism from $C^*(\CB, \CL,\theta, \CI_\af)$ to a $C^*$-algebra $A$ is injective if (and only if) its restriction to the canonical diagonal subalgebra $D:=C^*(\{s_{\af,A}s_{\af,A}^*: \af \in \CL^* ~\text{and}~ A \in \CI_\af\})$  of $C^*(\CB, \CL,\theta, \CI_\af)$  is injective. 
  However, it is quite challenging to check whether a Boolean dynamical system satisfies Condition (L), and the theorem may not be applicable in certain situations. Therefore, in this paper, we introduce a more generalized version of the Cuntz-Krieger uniqueness theorem that can be utilized without  Condition (L). 
   To be more precise, we define a specific subalgebra denoted by $M$ (see Definition \ref{dfn:abelian core}) within $C^*(\CB, \CL, \theta, \CI_\af)$  that possesses the following lifting property: a $*$-homomorphism of $C^*(\CB, \CL, \theta, \CI_\af)$ is injective if and only if it is injective when restricted to the subalgebra $M$.
        A generalized Cuntz-Krieger Uniqueness theorem has also been established for various classes, such as graph algebras (\cite[Theorem 3.12]{NR2012}), ultragraph algebras (\cite[Theorem 6.11]{CGW2020}), and higher rank graph algebras (\cite[Theorem 7.10]{BNR2014}), among others. 
   In each of the mentioned classes, a distinguished abelian subalgebra with a similar lifting property is defined. This subalgebra is referred to as the abelian core in both graph algebras and ultragraph algebras, and as the cycline subalgebra in higher rank graph algebras.        
   Returning to  our case,  the algebra $M$  is generated by the standard generators $s_{\af,A}s_{\bt,A}^*$ under the conditions where $\af=\bt$; or $\af=\bt\gm$ and $(\gm,A)$ forms a cycle with no exits; or $\bt=\af\gm$ and $(\gm,A)$ forms a cycle with no exits. 
   It's worth noting that the algebra $M$ is also commutative, and this fact will be demonstrated through a direct computation on the generators.
   We  refer to it as the abelian core of $C^*(\CB, \CL, \theta, \CI_\alpha)$.

To proceed with a proof of the uniqueness theorem,
 we use a groupoid model for a generalized Boolean dynamical system, namely,  the transformation groupoid $\mathbb{F} \ltimes_\varphi \partial E$ associated with the partial action $\Phi=(\{U_t\}_{t \in \mathbb{F}}, \{\varphi_t\}_{t \in \mathbb{F}})$ arising from a generalized Boolean dynamical system.
 First, we show that the spectrum of the canonical diagonal subalgebra $D$ of $C^*(\CB, \CL,\theta, \CI_\af)$ is homeomorphic to the boundary path space $\partial E$. Therefore,  the  diagonal subalgebra is isomorphic to the $C^*$-algebra $C_0(\partial E)$ on the unit space $\partial E$ of $\mathbb{F} \ltimes_\varphi \partial E$.
Next,  we describe the abelian core $M$ as the groupoid $C^*$-algebra of the interior of the isotropy group bundle   $\operatorname{Iso}(\mathbb{F} \ltimes_\varphi \partial E)$
of  $\mathbb{F} \ltimes_\varphi \partial E$. 
Then, by applying the uniqueness theorem (\cite[Theorem 3.1(b)]{BNRSW}) for general groupoid $C^*$-algebras  to our case, we  prove that if $\CB$ and $\CL$ are countable, then a  $*$-homomorphism of $C^*(\CB, \CL,\theta, \CI_\af)$ is injective if and only if it is injective on the abelian core. 
We should note that we assume that $\CB$ and $\CL$ are countable because \cite[Theorem 3.1(b)]{BNRSW} requires a groupoid to be second countable.

Lastly, we show that if the underlying Boolean dynamical system $(\CB, \CL,\theta)$ satisfies Condition (L), then  the diagonal subalgebra is a MASA, and the abelian core coincide with  the diagonal subalgebra. Hence, it follows that if $(\CB, \CL,\theta)$ satisfies Condition (L), then the abelian core is a MASA.

The paper is organized as follows. In Section \hyperref[preliminary]{2}, we present the relevant background,  notation and  some basic properties.   
In Section \hyperref[The diagonal subalgebra]{3}, we give a characterization of the canonical diagonal subalgebra.
In Section  \hyperref[GUT]{4}, we introduce the abelian core $M$ of $C^*(\CB, \CL,\theta, \CI_\af)$, and give a characterization of the the isotropy group bundle   $\operatorname{Iso}(\mathbb{F} \ltimes_\varphi \partial E)$. We then identify the abelian core $M$ with the  groupoid $C^*$-algebra of the interior of the isotropy group bundle   $\operatorname{Iso}(\mathbb{F} \ltimes_\varphi \partial E)$
of  $\mathbb{F} \ltimes_\varphi \partial E$, and prove our uniqueness theorem.
Additionally, we describe Condition (L) in terms of the groupoid $\mathbb{F} \ltimes_\varphi \partial E$ and the partial action  $\Phi=(\{U_t\}_{t \in G}, \{\varphi_t\}_{t \in G})$. We then prove that if $(\CB, \CL, \theta)$ satisfies Condition (L), then the abelian core is a MASA.

\section{Preliminaries}\label{preliminary} 

In this section,  we review some necessary background and give a few preliminary results (Lemma \ref{char:ultrafilter cycle}, \ref{range of path} and \ref{char2:ultrafilter cycle}) which will be used later. 

\subsection{Filters}
 Let $(P, \leq)$ be a partially ordered set with least element $0$. Given $X,Y \subseteq P$, we define 
\begin{align*}\uparrow X & :=\{b \in P: a \leq b ~\text{for some}~a \in X\}, \\
\downarrow X &:=\{b \in P: b \leq a ~\text{for some}~a \in X\},
\end{align*}
 and $\uparrow_Y X:= Y \cap \uparrow X$. For each $x \in P \setminus \{0\}$, we write $\uparrow x:=\uparrow \{x\}$. 
We say that $X$  is an  {\it upper set} if $\uparrow X=X$, and that  $X$ is a {\it lower set} if $\downarrow X=X$. We also say that  $X$ is {\it down-directed} if for all $x,y \in X$, there is $z \in X$ such that $z \leq x$ and $z \leq y$.
A non-empty subset  $\xi$ of $P$ is called a {\em filter}  if it is a  down-directed upper set with $0 \notin \xi$.  
A filter that is maximal  amomg filters with respect to inclusion is called an {\it ultrafilter}.

Let $\xi$ be a filter in a lattice $P$ with least elememt $0$. We say that $\xi$ is {\it prime} if for any $x,y \in P$, if $x \vee y \in\xi$, then $x\in\xi$ or $y\in\xi$.

\subsection{Boolean algebras}
 A \emph{Boolean algebra}  $\CB$ is a relatively complemented distributive lattice with least element $\emptyset$
(a Boolean algebras is so-called a {\it generalized Boolean algebra}).
For $A, B \in \CB$, the {\it meet} of $A$ and $B$ is denoted by $A \cap B$, the {\it join} of $A$ and $B$ is denoted by $A \cup B$, and the 
{\it relative complement} of $A$ relative to  $B$ is denoted by 
$B \setminus A$. 
The partial order is given by 
$$ A\subseteq B ~\text{if and only if}~ A \cap B =A$$
for $A, B \in \CB$. 
We say that   $A$ is  a  \emph{subset of} $B$ if $A \subseteq B$.

 A non-empty subset $\CI\subseteq \CB$ is called an {\it ideal} if it is closed under finite unions, that is, $A\cup B\in\CI$ whenever $A,B\in\CI$, and it is a lower set. Every ideal of $\CB$ is again a Boolean algebra.

Let $\xi$ be a filter in a Boolean algebra $\CB$. It is well-known that $\xi$ is an ultrafilter if and only if it is prime, if and only if $\{A \in \CB: A \cap B \neq \emptyset ~\text{for all}~ B \in \xi\} \subseteq \xi$.
We denote by $\widehat{\CB}$ the set of all ultrafilters of $\CB$. 
For each $A \in \CB$, we let $$Z(A):=\{\xi\in\widehat{\CB}:A\in\xi\}.$$ 
Then, the the family $\{Z(A): A \in \CB \}$ is a basis of compact-open sets for a Hausdorff topology on $\widehat{\CB}$. 
We call the set $\widehat{\CB}$ equipped with this topology the {\it Stone dual} of $\CB$.

\subsection{Boolean dynamical systems}

Let $\CB$ and $\CB'$ be Boolean algebras. 
A map $\phi: \CB \rightarrow \CB'$  is  called a {\em Boolean homomorphism} if $$\phi(A \cap B)=\phi(A) \cap \phi(B), \phi(A \cup B)=\phi(A) \cup \phi(B) ~\text{and}~\phi(A \setminus B)=\phi(A) \setminus \phi(B)$$ for all $A,B \in \CB$. A  map $\theta: \CB \rightarrow \CB $ is called an {\it action} on  $\CB$ if it is a Boolean homomorphism with $\theta(\emptyset)=\emptyset$.

Let $\CL$ be a set and $n \in \N$. We define $\CL^n:=\{(\af_1, \dots, \af_n): \af_i \in \CL\}$ and write $\af_1 \dots \af_n$ instead of  $(\af_1, \dots, \af_n) \in \CL^n$.
For a word $\af=\af_1 \dots \af_n \in \CL^n$, its length is denoted by $|\af|$, namely $|\af|=n$. The set $\CL^{\geq 1}=\cup_{n \geq 1} \CL^n$ is the set of all words of positive finite length. We also let  $\CL^*:=\cup_{n \geq 0} \CL^n$, where $\CL^0:=\{\emptyset \}$.
  Similarly, we define  the set of all infinite words by $\CL^\infty=\{(\af_i)_{i \in \N}: \af_i \in \CL \}$ and let 
$\CL^{\leq \infty}:= \CL^* \cup \CL^\infty$. For $\af =(\af_i)_{i \in \N} \in \CL^\infty$, we write $|\af|=\infty$. 

Let $\af \in \CL^*$ and $\bt \in \CL^{ \leq \infty}$. We denote by $\af\bt$ the 
   the concatenation of $\af$ and $\bt$, where we mean $\emptyset \bt=\bt$ and $\af \emptyset=\af$. For $k\in\N$, $\af^k$ is a k-times concatenation of $\af$ and $\af^\infty$ is a infinitely many times concatenation of $\af$, and let $\af^0:=\emptyset$. 
      For $1\leq i\leq j\leq |\bt|$, we let  $\bt_{i,j}=\bt_i\cdots \bt_j$ if $j < \infty$ and $\bt_{i,j}=\bt_i \bt_{i+1}\cdots $ if $j = \infty$. If $j < i$, let $\bt_{i,j} =\emptyset$.
      For $\af,\bt \in \CL^*$, we say that $\af$ is a {\it beginning} of $\bt$ if $\bt=\af\bt'$ for some $\bt'  \in \CL^*$. We say that $\af$ and $\bt$ are {\it comparable} if $\af$ is a beginning of $\bt$ or $\bt$ is a beginning of $\af$.

   A {\em Boolean dynamical system}   (\cite[Definition 2.1]{CaK2})  is a triple $(\CB,\CL,\theta)$ where $\CB$ is a Boolean algebra, $\CL$ is a set, and $\{\theta_\af\}_{\af \in \CL}$ is a set of actions on $\CB$.

   Let  $(\CB, \CL,\theta)$ be a Boolean dynamical system. For $\af=\af_1 \cdots \af_n \in \CL^*$,  the action $\theta_\af: \CB \rightarrow \CB$ is defined as $\theta_\af:=\theta_{\af_n} \circ \cdots \circ \theta_{\af_1}$, where $\theta_\emptyset:=\text{Id}$. 
      For $B \in \CB$, we define $\Delta_B:=\{\af \in \CL:\theta_\af(B) \neq
\emptyset \} $.
We say that $A \in \CB$ is {\em regular} (\cite[Definition 3.5]{COP})
if for any $\emptyset \neq B \subseteq A$, we have $0 < |\Delta_B| < \infty$.
We write $\CB_{reg}$ for the
set of all regular sets where we will include $\emptyset$.

 Let $\af=\af_1 \cdots \af_{n} \in \CL^{\geq 1}$,  $A \in \CB\setminus\{\emptyset\}$, and $\eta \in \widehat{\CB}$. A pair $(\af,A)$ is  a \emph{cycle} (\cite[Definition 9.5]{COP}) if $B=\theta_\af(B)$ for $B \subseteq A$.  A cycle $(\af,A)$  has \emph{no exits} (\cite[Definition 9.5]{COP}) if for $t\in\{1,2,\dots,n\}$ and $ \emptyset \neq B \subseteq \theta_{\beta_{1,t}}(A)$, we have $B \in \CB_{reg}$ with $\Delta_{B}=\{\af_{t+1}\}$ (where  $\af_{n+1}:=\af_1$).
 We say that $(\CB,\CL,\theta)$ satisfies \emph{Condition (L)} (\cite[Definition 9.5]{COP}) if it has no cycle with no exits.
 A pair $(\af,\eta)$ is an  \emph{ultrafilter cycle} (\cite[Definition 3.1]{CaK1}) if $\theta_\af(A)\in\eta$ for all $A\in\eta$.

\begin{lem}\label{char:ultrafilter cycle} Let $\af \in \CL^{\geq 1}$ and $\eta \in \widehat{\CB}$. Then $(\af, \eta)$ is an ultrafilter cycle if and only if $A \cap \theta_{\af}(A) \neq \emptyset$ for all $A \in \eta$.
\end{lem}

\begin{proof} ($\Rightarrow$) It is clear.

($\Leftarrow$) We first claim that $A \cap \theta_{\af}(A) \in \eta$ for all $A \in \eta$. To prove it, it is enough to show that $(A \cap \theta_{\af}(A)) \cap B \neq \emptyset$ for all $B \in \eta$. Let $B \in \eta$. Then $A \cap B \in \eta$, so $(A \cap B) \cap \theta_{\af}(A \cap B) \neq \emptyset$. Since  $(A \cap B) \cap \theta_{\af}(A \cap B) \subseteq (A \cap B) \cap \theta_{\af}(A) $, we have $ (A \cap \theta_{\af}(A)) \cap B \neq \emptyset$.

Now, for all $A \in \eta$, we have $\eta \ni A \cap \theta_{\af}(A) \subseteq \theta_{\af}(A)$, and hence,  $\theta_{\af}(A) \in \eta$. So, $(\af, \eta)$ is an ultrafilter cycle. 
\end{proof}

\subsection{Generalized Boolean dynamical systems and their $C^*$-algebras}\label{GBDS}

A {\em generalized Boolean dynamical system} (\cite[Definition 3.2]{CaK2}) is a quadruple  $(\CB,\CL,\theta,\CI_\alpha)$ where  $(\CB,\CL,\theta)$ is  a Boolean dynamical system  and  $\{\CI_\alpha:\alpha\in\CL\}$ is a family of ideals in $\CB$ such that $ \theta_\af(\CB) \subseteq\CI_\alpha$ for each $\alpha\in\CL$.

\begin{dfn}\label{def:representation of GBDS} (\cite[Definition 3.3]{CaK2})
	Let $(\CB,\CL,\theta, \CI_\af)$ be a  generalized Boolean dynamical system. A 
	\emph{$(\CB,\CL,\theta, \CI_\af)$-representation} in a $C^*$-algebra $A $
		is   families of projections $\{P_A:A\in\mathcal{B}\}$ and  partial isometries $\{S_{\alpha,B}:\alpha\in\mathcal{L},\ B\in\mathcal{I}_\alpha\}$ 
	that satisfy  
	\begin{enumerate}
		\item[(i)] $P_\emptyset=0$, $P_{A\cap A'}=p_Ap_{A'}$, and $P_{A\cup A'}=P_A+P_{A'}-P_{A\cap A'}$ for $A,A'\in\mathcal{B}$;
		\item[(ii)] $P_AS_{\alpha,B}=S_{\alpha,  B}P_{\theta_\af(A)}$  for $A\in\mathcal{B}$, $\alpha \in\mathcal{L}$ and $B\in\mathcal{I}_\alpha$;
		\item[(iii)] $S_{\alpha,B}^*S_{\alpha',B'}=\delta_{\alpha,\alpha'}P_{B\cap B'}$ for  $\alpha,\alpha'\in\mathcal{L}$, $B\in\mathcal{I}_\alpha$ and $B'\in\mathcal{I}_{\alpha'}$;
		\item[(iv)] $P_A=\sum_{\af \in\Delta_A}S_{\af,\theta_\af(A)}S_{\af,\theta_\af(A)}^*$ for   $A\in \mathcal{B}_{reg}$. 
	\end{enumerate}
	Given a $(\CB, \CL, \theta, \CI_\af)$-representation $\{P_A, S_{\af,B}\}$ in a $C^*$-algebra $\CA$, we denote by $C^*(P_A, S_{\af,B})$ the $C^*$-subalgebra of $\CA$ generated by $\{ P_A,  S_{\af,B}\}$.

	 \end{dfn}

Put $\CI_\emptyset := \CB$. For $\af=\af_1\af_2 \cdots \af_n \in \CL^{\geq 1}$, we define $\CI_\af$ to be the smallest ideal of $\CB$ that contains $\theta_{\af_2 \cdots \af_n}(\CI_{\af_1})$, that is, 
\begin{align*}
	\CI_\af:=\{A \in \CB : A \subseteq \theta_{\af_2 \cdots \af_n}(B)~\text{for some }~ B \in \CI_{\af_1}\}.
\end{align*}
Given  $A\in\CB$ and  $\af=\af_1\af_2 \cdots \af_n \in \CL^{\geq 1}$, we set  $S_{\emptyset, A}:=P_A$ and 	
\[
S_{\af,A}:=S_{\af_1,B}S_{\af_2, \theta_{\af_2}(B)}S_{\af_3, \theta_{\af_2\af_3}(B)} \cdots S_{\af_n,A},
\] 
where $B \in \CI_{\af_1}$ is such that $A \subseteq \theta_{\af_2 \cdots \af_n}(B)$.

\begin{remark}\label{computations of generators} Let $\{P_A, S_{\alpha,B}: A\in\CB, \alpha\in\CL ~\text{and}~ B\in\CI_\alpha\}$ be a $(\CB,\CL,\theta,\CI_\alpha)$-representation.
\begin{enumerate}
\item (\cite[Lemma 3.9]{CaK2}) For $\af, \bt \in \CL^*$, $A \in \CI_\af$ and $B \in \CI_\bt$, we have
\[
S_{\af,A}^*S_{\bt,B}= \left\{ 
\begin{array}{ll}
    P_{A \cap B} & \hbox{if\ }\af =\bt\\
    S_{\af', A \cap \theta_{\af'}(B)}^* & \hbox{if\ }\af =\bt\af' \\
    S_{\bt',B \cap \theta_{\bt'}(A)}   & \hbox{if\ } \bt=\af\bt' \\
    0 & \hbox{otherwise.} \\
\end{array}
\right.
\]
\item  We then have
that
\begin{align*}
C^*(P_A,S_{\alpha,B})
=\overline{{\rm \operatorname{span}}}\{S_{\af,A}S_{\bt,A}^*: \af,\bt
\in \CL^* ~\text{and}~ A \in \CI_\af\cap \CI_\bt \}.\label{eq:3}
\end{align*}
\item By (1), it easily follows that  
\[
(S_{\af,A}S_{\af,A}^*)(S_{\bt,B}S_{\bt,B}^*)= \left\{ 
\begin{array}{ll}
    S_{\af, A \cap B}S_{\af, A \cap B}^* & \hbox{if\ }\af=\bt \\
    S_{\af, A \cap \theta_{\af'}(B)}S_{\af,  A \cap \theta_{\af'}(B)}^*    & \hbox{if\ }\af=\bt\af'\\
    S_{\bt, B \cap \theta_{\bt'}(A)}S_{\bt, B \cap \theta_{\bt'}(A )}^*                  & \hbox{if\ } \bt=\af\bt' \\
    0   & \hbox{otherwise.} 
\end{array}
\right.
\]
\item By (3), it is easy to see that the elements of the form $S_{\af,A}S_{\af,A}^*$ are commuting projections. 
\end{enumerate}
\end{remark}

We denote by $C^*(\CB, \CL,\theta, \CI_\af)$ the $C^*$-algebra generated by a universal 
$(\CB,\CL,\theta, \CI_\af)$-representation $\{p_A, s_{\af,B}\}$. Let 
$$G_{(\CB, \CL,\theta, \CI_\af)}:=\{s_{\af,A}s_{\bt,A}^*: \af,\bt \in \CL^*, A \in \CI_\af \cap \CI_\bt\}$$
and we call  the elements $s_{\af,A}s_{\bt,A}^*$ in the above set  the {\it standard generators}.

\begin{dfn} Let $(\CB, \CL, \theta, \CI_\af)$ be a generalized Boolean dynamical system.
The {\it diagonal subalgebra $D(\CB, \CL,\theta, \CI_\af)$} of $C^*(\CB, \CL,\theta, \CI_\af)$ is the subalgebra of $C^*(\CB, \CL,\theta, \CI_\af)$ generated by the commuting projections $s_{\af,A}s_{\af,A}^*$, that is,
$$D(\CB, \CL,\theta, \CI_\af)=C^*(\{s_{\af,A}s_{\af,A}^*: \af \in \CL^* ~\text{and}~ A \in \CI_\af\}).$$
\end{dfn}

By Remark \ref{computations of generators}(3), it is easy to see that $$D(\CB, \CL,\theta, \CI_\af)=\overline{span}\{s_{\af,A}s_{\af,A}^*: \af \in \CL^* ~\text{and}~ A \in \CI_\af\}$$and   it is a commutative $C^*$-algebra. We simply write $D:=D(\CB, \CL,\theta, \CI_\af)$ if there is no confusion.

\subsection{topological graph}
Let $(\CB, \CL,\theta, \CI_\af)$ be a generalized Boolean dynamical system.
 Following \cite{CasK1}, we let  $\CW^{\geq 1}=\{\alpha\in\CL^{\geq 1}:\CI_\alpha\neq \{\emptyset\}\}$,  $\CW^*=\{\alpha\in\CL^*:\CI_\alpha\neq \{\emptyset\}\}$, $\CW^{\infty}=\{\alpha\in\CL^{\infty}:\alpha_{1,n}\in\CW^{\geq 1} ~\text{for all}~ n \geq 1\}$ and $\CW^{\leq\infty}=\CW^*\cup \CW^{\infty}$.
 We let 
 $$X_\emptyset:=\WCB ~\text{and}~X_\af:=\widehat{\CI_\af} $$
 for each $\af \in \CW^*$.
 We equip $X_\emptyset$ with the tolopogy given by the basis $\{Z(A):A\in\CB\}$ and $X_\af$
 with the topology generated by $\{Z(\af, A): A\in\CI_\af\}$, where 
  we let $Z(\af, A):=\{\CF \in X_\af: A \in \CF\}$ for $A \in \CI_\af$. 
  We also  equip  $X_\emptyset\cup\{\emptyset\}$  with a suitable topology; if $\CB$ is unital, the topology is such that $\{\emptyset\}$ is an isolated point. If $\CB$ is not unital, then $\emptyset$ plays the role of the point at infinity in the one-point compactification of $X_\emptyset$.

Let $\af, \bt \in \CW^*$ be such that $\af\bt \in \CW^*$. Consider an open  subspace $X_{(\af)\bt}$ of $X_\bt$ defined by 
$$X_{(\af)\bt}:=\{\CF \in X_\bt:\CF \cap \CI_{\af\bt}\neq \emptyset  \}.$$ Note that $X_{(\emptyset)\emptyset}=X_{\emptyset}$ and $X_{(\emptyset)\bt}=X_\bt$. 
Then there are a continuous map 
 $$g_{(\af)\bt}: X_{(\af)\bt} \to X_{\af\bt}$$ given by $$  g_{(\af)\bt}(\CF):=     \CF \cap \CI_{\af\bt}  $$
for  $\CF \in X_{(\af)\bt}$, and 
 a continuous map 
$$h_{[\af]\bt}:X_{\af\bt} \to X_{(\af)\bt}$$ given by $$h_{[\af]\bt}(\CF):=\uparrow_{\CI_\bt}\CF$$ for  $\CF \in X_{\af\bt}$.
Note that  $g_{(\emptyset)\emptyset}$ and $h_{[\emptyset]\emptyset}$ are the identity functions on $X_\emptyset$, and that   $g_{(\emptyset)\bt}$ and $h_{[\emptyset]\bt}$ are the identity functions on $X_\bt$. Also, we notice that   $h_{[\af]\bt}: X_{\af\bt} \to X_{(\af)\bt}$ and $g_{(\af)\bt}:X_{(\af)\bt} \to X_{\af\bt}$ are mutually inverses.

Let $\af, \bt \in \CW^* \setminus \{\emptyset\}$ be such that $\af\bt \in \CW^*$. We also have a continuous map
$$f_{\af[\bt]}: X_{\af\bt} \to X_\af $$ given by $$
 f_{\af[\bt]}(\CF)=\{A \in \CI_\af: \theta_\bt(A) \in \CF\}$$
 for $\CF \in X_{\af\bt}$. For $\af=\emptyset$, the map $f_{\emptyset[\bt]}: X_\bt \to X_\emptyset \cup \{\emptyset\} $ defined as above is also continuous (\cite[Lemma 3.23]{CasK1}).

 Now, we let
  \begin{align*}
 E^0&:=X_\emptyset, \\
 F^0&:=X_\emptyset\cup\{\emptyset\} \\
 E^1&:=\bigl\{e^\alpha_\eta:\alpha\in\CL,\ \eta\in X_\alpha\bigr\}
 \end{align*}
 and we 
 equip $E^1$ with the topology generated by 
$\bigcup_{\alpha\in\CL} \{Z^1(\af, B):B\in\CI_\alpha\}, $
where $$Z^1(\af, B):=\{e^\af_\eta: \eta \in X_\af  , B \in \eta\}.$$  Note that $E^1_{(\CB,\CL,\theta,\CI_\alpha)}$ is homeomorphic to the disjoint union of the family $\{X_{\af}\}_{\af\in\CL}$.

We define the maps
 $d:E^1\to E^0$ and $r:E^1\to F^0$ by 
\[
	d(e^\alpha_\eta)=h_{[\alpha]\emptyset}(\eta) \text{ and } r(e^\alpha_\eta)=f_{\emptyset[\af]}(\eta).
\]
 Then, $(E^1, d,r)$ is a topological correspondence from $E^0$ to $F^0$ (\cite[Proposition 7.1]{CasK1}).

Given $n \geq 2$,  we define a space  of paths with length $n$ by
\[E^{n}:=\{(e^{\af_1}_{\eta_1},\ldots,e^{\af_n}_{\eta_n}) \in \prod_{i=1}^n E^1:d(e^{\af_i}_{\eta_i})=r(e^{\af_{i+1}}_{\eta_{i+1}}) ~\text{for}~1\leq i<n\},\]
where we equip it with a subspace topology of the product space $\prod_{i=1}^n E^1$.   We  usually write $e^{\af_1}_{\eta_1} \cdots e^{\af_n}_{\eta_n}$ for a path  $(e^{\af_1}_{\eta_1},\ldots,e^{\af_n}_{\eta_n}) 
\in E^n$.
We define  the {\it finite path space} by $E^*:= \sqcup_{n=0}^{\infty} E^n$ endowed with the disjoint union topology, and  the {\it infinite path space} by
\[E^{\infty}:=\{(e^{\af_i}_{\eta_i})_{i\in\mathbb{N}}\in \prod_{i=1}^\infty E^1:d(e^{\af_i}_{\eta_i})=r(e^{\af_{i+1}}_{\eta_{i+1}}) ~\text{for}~ i\in\mathbb{N}\}.\]
For $\xi \in E^0$, we let $r(\xi)=d(\xi)=\xi$.  
For a path $\mu=e^{\af_1}_{\eta_1} \cdots e^{\af_n}_{\eta_n} \in E^n$, we let $d(\mu)=d(e^{\af_n}_{\eta_n})$ and $r(\mu)=r(e^{\af_1}_{\eta_1})$ if $n\geq 1$. For an infinite path $\mu=(e^{\af_i}_{\eta_i})_{i\in\mathbb{N}}$, we define $r(\mu)=r(e^{\af_1}_{\eta_1})$.
  We denote by $|\mu|$ the length of a path $\mu \in E^* \sqcup E^\infty$ and regard a vertex in $E^0$ as a path of length 0.
For  $1\leq i\leq j\leq |\mu|$, we also denote by $\mu_{[i,j]}$ the sub-path $e^{\af_i}_{\eta_i} \cdots e^{\af_j}_{\eta_j}$ of  $\mu=(e^{\af_n}_{\eta_n})_{n\in\mathbb{N}}$, where $\mu_{[i,i]}=e^{\af_i}_{\eta_i}$. If $j < i$, set $\mu_{[i,j]} =\emptyset$.

For $\mu \in E^*$ and $\nu \in E^* \sqcup E^\infty$, we write $\mu\nu$  for the concatenation of $\mu$ and $\nu$. For   $\mu \in E^* \setminus E^0$ and $n \in \N$, 
$\mu^n$ represents $\mu$ concatenated $n$ times and $\mu^\infty$ is $\mu$ concatenated infinitely many times.

A finite path $\mu \in E^* \setminus E^0$ is called a {\it loop} if $r(\mu) = d(\mu)$. The vertex $r(\mu)$
is called the {\it base point} of $\mu$. A loop $\mu=e^{\af_1}_{\eta_1} \cdots e^{\af_n}_{\eta_n}$ is said to be {\it without entrances} if $r^{-1}(r(e^{\af_i}_{\eta_i}))=\{e^{\af_i}_{\eta_i}\}$ for $i=1, \cdots, n$.

\begin{lem} \label{range of path} For $\mu=e^{\af_1}_{\eta_1} \cdots e^{\af_n}_{\eta_n} \in E^*$, we have 
\begin{align*} 
r(\mu)&=f_{\emptyset[\af_1 \cdots \af_n]}(g_{(\af_1 \cdots \af_{n-1})\af_n}(\eta_n))\\
&=f_{\emptyset[\af_1 \cdots \af_n]}(g_{(\af_1 \cdots \af_n)\emptyset}(\eta)),
\end{align*}
where we put $\eta:=h_{[\af_n]\emptyset}(\eta_n) \in \widehat{\CB}$.
\end{lem}

\begin{proof}
We see that
\begin{align*} r(\mu)&=f_{\emptyset[\af_1 \cdots \af_n]}(g_{(\af_1 \cdots \af_{n-1})\af_n}(\eta_n))\\
& =f_{\emptyset[\af_1 \cdots \af_n]}(g_{(\af_1 \cdots \af_{n-1})\af_n}(g_{(\af_n)\emptyset}(h_{[\af_n]\emptyset}(\eta_n) )) \\
&=f_{\emptyset[\af_1 \cdots \af_n]}(g_{(\af_1 \cdots \af_n)\emptyset}(h_{[\af_n]\emptyset}(\eta_n))) \\
&=f_{\emptyset[\af_1 \cdots \af_n]}(g_{(\af_1 \cdots \af_n)\emptyset}(\eta)),
\end{align*}
where the first equation follows from \cite[Lemma 3.3]{CasK2}.
\end{proof}

\begin{lem}\label{char2:ultrafilter cycle} Let $\mu=e^{\af_1}_{\eta_1} \cdots e^{\af_n}_{\eta_n} \in E^*$ and put $\eta:=h_{[\af_n]\emptyset}(\eta_n) \in \widehat{\CB}$. Then $(\af_1 \cdots \af_n, \eta )$ in an ultrafilter cycle if and only if $\mu$ is a loop at $\eta$. 
\end{lem}

\begin{proof} 
    $(\Rightarrow)$    Choose $A \in \eta$. Since $(\af_1 \cdots \af_n, \eta)$ in an ultrafilter cycle, we have $\theta_{\af_1 \cdots \af_n}(A) \in \eta \cap \CI_{\af_1 \cdots \af_n}$, so, $A \in f_{\emptyset[\af_1 \cdots \af_n]}(g_{(\af_1 \cdots \af_n)\emptyset}(\eta))$. Thus, $\eta \subseteq  f_{\emptyset[\af_1 \cdots \af_n]}(g_{(\af_1 \cdots \af_n)\emptyset}(\eta))$. Then, we have $\eta =  f_{\emptyset[\af_1 \cdots \af_n]}(g_{(\af_1 \cdots \af_n)\emptyset}(\eta)) (=r(\mu))$ since both are ultrafilters.
    
    Hence,  $r(\mu)=\eta=d(\mu)$.  So,  $\mu$ is a loop at $\eta$.

$(\Leftarrow)$ Since $\mu$ is a loop, we have 
$$f_{\emptyset[\af_1 \cdots \af_n]}(g_{(\af_1 \cdots \af_n)\emptyset}(\eta))=r(\mu)=d(\mu)=\eta.$$ It means that $(\af_1 \cdots \af_n, \eta )$ in an ultrafilter cycle.
\end{proof}

\subsection{The boundary path space $\partial E$}
Let $E=(E^1,d,r)$ be a topological correspondence  from $E^0$ to $F^0$. Define 
$F^0_{sce}  := F^0 \setminus \overline{r(E^1)}$, 
$F^0_{fin}:=\{v \in F^0: \exists V~\text{neighborhood of}~v ~\text{such that}~ r^{-1}(V)~\text{is compact} \}$, $F^0_{rg}:=F^0_{fin} \setminus \overline{F^0_{sce}}$, and $F^0_{sg} :=F^0 \setminus F^0_{rg}$ (\cite[Section~1]{Ka2004}). Also, we let  $E^0_{rg}:=F^0_{rg}\cap E^0$ and $E^0_{sg}:=F^0_{sg}\cap E^0$.

\begin{dfn}(\cite[Definition 7.5]{CasK1})\label{def:boundary path space}  Let  $(\CB,\CL,\theta, \CI_\af)$ be a generalized Boolean dynamical system and $E=(E^1,d,r)$ be the associated  topological correspondence. Define the {\it boundary path space} of $E$ by 
$$\partial E :=E^{\infty} \sqcup \{\mu\in E^* : d(\mu)  \in E^0_{sg}\}.$$
For a subset $S \subset E^*$, we define
$$\CZ(S)=\{\mu \in \partial E:~\text{either} ~r(\mu) \in S, ~\text{or there exists}~  1 \leq i \leq |\mu| ~\text{such that}~ \mu_{[1,i]} \in 
S \}.$$
We equip $\partial E$ with the topology generated by the basic open sets $\CZ(U)\cap \CZ(K)^c$, where $U$ is an open set of $E^*$ and $K$ is a compact set of $E^*$.
\end{dfn}

We note that $\partial E$ is a locally compact Hausdorff space. 
 For $(e^{\af_k}_{\eta_k})_{k=1}^n\in E^*$, where $1\leq n$, we have $\af_1\cdots\af_n\in\CW^*$ by \cite[Lemma 7.9]{CasK1}.
 Thus, we can define  a map $\CP:\partial E\to\CW^{\leq\infty}$  given by $\CP((e^{\af_k}_{\eta_k})_{k=1}^N)=(\af_k)_{k=1}^N$, where $1\leq N\leq\infty$, and $ \CP(\eta)=\emptyset$ for $\eta \in E^0$.
For $\alpha=\af_1 \cdots \af_{|\af|} \in\mathcal{W}^{\geq 1}$ and   $A  \in \CI_\af$, we  define the cylinder set by 
\begin{align*}\CN(\af,A)&:=\{(e_{\eta_n}^{\beta_n}) \in \partial E: \CP((e_{\eta_n}^{\beta_n}))_{1,|\af|}=\af ~\text{and}~   A\in \eta_{|\alpha|}\}\\
&=\{(e_{\eta_n}^{\beta_n}) \in \partial E: \bt_1 \cdots \bt_{|\af|}=\af ~\text{and}~   A\in \eta_{|\alpha|}\}.
\end{align*}
Also, for  $A\in\CB$, we let 
\begin{align*}\CN(\emptyset,A):=\{\mu \in\partial E: A\in r(\mu)\} .
\end{align*}
Note that  $\CN(\af,\emptyset)=\emptyset$ for  $\af\in\CW^*$.
Then, we see that 
the sets $\CN(\af,A)$
are compact-open sets in $\partial E$ for each $\alpha\in\mathcal{W}^*$ and $A \in \CI_\af$ .

\begin{lem} (\cite[Lemma 3.8 and 3.18]{CasK2}) Let $\partial E$ be the boundary path space of E.
\begin{enumerate}
\item For $\alpha\in\mathcal{W}^*$ and $A \in \CI_\af$,  the sets $\CN(\af,A)$
are compact-open sets in $\partial E$.
\item The $C^*$-algebra $C_0(\partial E)$ is generated by the set 
$$\{1_{\CN(\af,A)}: \af \in \CW^* ~\text{and}~ A \in \CI_\af\},$$
where  we  denote by $1_{\CN(\af,A)}$ the characteristic function on $\CN(\af,A)$.
\end{enumerate}
\end{lem}

\subsection{A partial action on $\partial E$  and the partial crossed product  $C_0(\partial E) \rtimes_{\hat{\varphi}} \mathbb{F}$}
Let $\mathbb{F}$ be the free group generated by $\CL$. We identify the identity of $\mathbb{F}$ with $\emptyset$ and  $\CW^*$ is considered as a subset of $\mathbb{F}$.

Define
\begin{itemize}
\item $U_\emptyset:=\partial E$;
\item for $\af, \bt \in \CW^*$  such that $\CI_\af \cap \CI_\bt \neq \emptyset$, let
\[ U_{\af\bt^{-1}}  :=\{(e^{\mu_k}_{\eta_k})_{k = 1}^{N}\in\partial E :\mu_1 \cdots \mu_{|\af|}=\af \text{ and }\eta_{|\af|}\cap\CI_{\bt}\neq\emptyset \};\]
here, to easy notations, we put
\begin{align*} 
U_\emptyset &:= U_{\emptyset\emptyset};\\
U_{\af}&:=U_{\af\emptyset}=\{(e^{\mu_k}_{\eta_k})_{k=1}^N\in\partial E : \gm_1 \cdots \gm_{|\af|}=\af \}; \\
U_{\bt^{-1}}&:=U_{\emptyset \bt^{-1}} =\{\mu \in\partial E : r(\mu) \cap \CI_\bt \neq \emptyset\};
\end{align*}
 \item for all the other elements $\gm \in \mathbb{F}$, let $U_\gm =\emptyset$. 
   \end{itemize}
We define

\begin{itemize}
\item $\varphi_\emptyset: U_\emptyset \to U_\emptyset$ as the identity map;
\item for $\af, \bt \in \CW^*$ such that $\alpha\beta^{-1}$ is in reduced form in $\mathbb{F}$ and $U_{\bt\af^{-1}}\neq\emptyset$, we define a map $$\varphi_{\af\bt^{-1}}:U_{\bt\af^{-1}}\to U_{\af\bt^{-1}}$$
by 
$$\varphi_{\af\bt^{-1}}(\mu):=e^{\af_1}_{\eta_1}\ldots e^{\af_n}_{\eta_{n}}\mu_{[|\bt|+1,|\mu|]} $$ for $\mu =(e^{\mu_i}_{\xi_i})_{i = 1}^N\in U_{\bt\af^{-1}}$,  where  
\begin{align*} \eta_n &:=g_{(\af_n)\emptyset}(h_{[\bt_{|\bt|}]\emptyset}(\xi_{|\bt|})) , \\
\eta_{n-1}&:=g_{(\af_{n-1})\emptyset}(f_{\emptyset[\af_n]}(\eta_n)), \\
\eta_{n-2}&:=g_{(\af_{n-2})\emptyset}(f_{\emptyset[\af_{n-1}]}(\eta_{n-1})), \\
& \vdots \\
\eta_2 &:= g_{(\af_2)\emptyset}(f_{\emptyset[\af_3]}(\eta_3)),\\
\eta_1&:= g_{(\af_1)\emptyset}(f_{\emptyset[\af_2]}(\eta_2));
\end{align*}
\item for all the other elements $\gm \in \mathbb{F}$, define $\varphi_\gm: U_{\gm^{-1}} \to U_\gm$ as the empty map. 
\end{itemize}
Then,  $\Phi:=(\{U_t\}_{t \in \mathbb{F}}, \{\varphi_t\}_{t \in \mathbb{F}})$ is a semi-saturated orthogonal  partial action of  $\mathbb{F}$ on $\partial E$ (\cite[Proposition 3.6]{CasK1}).
For $t \in \mathbb{F}$, define $\hat{\varphi}_t: C_0(U_{t^{-1}}) \to C_0(U_t)$ by $$\hat{\varphi}_t(f)=f \circ \varphi_{t^{-1}}.$$
Then, $(\{C_0(U_t)\}_{t \in \mathbb{F}}, \{\hat{\varphi}_t\}_{t \in \mathbb{F}} )$ is a $C^*$-algebraic partial dynamical system. Thus, we have  the partial crossed product 
$$C_0(\partial E) \rtimes_{\hat{\varphi}} \mathbb{F} = \overline{\operatorname{span}} \Big\{ \sum_{t \in \mathbb{F}}f_t\dt_t: f_t \in C_0(U_t) ~\text{and}~ f_t \neq 0  ~\text{for finitely many}~ t \in \mathbb{F} \Big\},$$
where   the 
multiplication and involution   in $C_0(\partial E) \rtimes_{\hat{\varphi}} \mathbb{F}$ are given by 
\begin{align*} &(a\dt_s)(b \dt_t)=\hat{\varphi}_s(\hat{\varphi}_{s^{-1}}(a)b)\dt_{st}, ~\text{and}~ \\
&(a\dt_s)^*=\hat{\varphi}_{s^{-1}}(a)\dt_{s^{-1}},
\end{align*}
and the closure is with respect to the universal norm.  We can think that $\dt_t$ serves a place holder to indicate the coordinate of $f_t\in C_0(U_t)$.

We let 
$C^*(\{1_{\CN(\emptyset,A)}\dt_\emptyset, 1_{\CN(\af,B)}\dt_\af\}) \subseteq C_0(\partial E) \rtimes_{\hat{\varphi}} \mathbb{F} $ denote the $C^*$-subalgebra generated by $\{1_{\CN(\emptyset,A)}\dt_\emptyset, 1_{\CN(\af,B)}\dt_\af: A \in \CB ~\text{and}~\af \in \CL, B \in \CI_\af \}$. Then we have by \cite[Proposition 3.21]{CasK1} that 
$C_0(\partial E) \rtimes_{\hat{\varphi}} \mathbb{F} = C^*(\{1_{\CN(\emptyset,A)}\dt_\emptyset, 1_{\CN(\af,B)}\dt_\af\})$
and there is an isomorphism $\psi: C^*(\CB,\CL,\theta,\CI_\af) \to C_0(\partial E) \rtimes_{\hat{\varphi}} \mathbb{F}$
given by $$\psi(p_A)=1_{\CN(\emptyset,A)}\dt_\emptyset ~\text{and}~\psi(s_{\af,B})=1_{\CN(\af,B)}\dt_\af$$ for $A \in \CB$, $\af \in \CL$ and $B \in \CI_\af$  (see  \cite[Corollary 4.5 and Remark 4.6]{CasK1}).

\subsection{Groupoids} 
Let $\CG$ be a locally compact Hausdorff groupoid. We say that $\CG$ is {\it \'etale} if the range and source maps $r,s: \CG \to \CG^{(0)}$ are local homeomorphism. 
A {\it bisection} in $\CG$ is a set $U \subseteq \CG$ such that $r|_U$ and $s|_U$ are homeomorphisms on $U$. 
An \'etale groupoid is {\it ample} if it has a basis of compact open bisections. 
It is known in \cite{Exel2010} that $\CG$ is ample if and only if its unit space $\CG^{(0)}$ is totally disconnected. 

For a unit  $u \in \CG^{(0)}$, we let 
$\CG^u:=r^{-1}(u)$, $\CG_u:=s^{-1}(u)$ and $\CG^u_u:=r^{-1}(u) \cap s^{-1}(u)$.
The {\it isotropy group bundle} of $\CG$ is $$\operatorname{Iso}(\CG):=\bigcup_{u \in \CG^{(0)}}\CG_u^u=\{\gm \in \CG: s(\gm)=r(\gm)\}.$$
Note that $\operatorname{Iso}(\CG)$ is closed in $\CG$. We write $\operatorname{Iso}(\CG)^{\circ}$ for the interior of $\operatorname{Iso}(\CG)$ in $\CG$. If $\CG$ is \'etale, then $\CG^{(0)} \subseteq \operatorname{Iso}(\CG)^{\circ}$  and $\operatorname{Iso}(\CG)^{\circ}$ is an  \'etale subgroupoid of $\CG$.

A unit $u \in\CG^{(0)}$ is said to have {\it trivial isotropy} if $\CG_u^u=\{u\}$. 
We say that $\CG$ is {\it topologically principal} if the set of units
 with trivial isotropy is dense in $\CG^{(0)}$, and that $\CG$ is {\it effective} if $\operatorname{Iso}(\CG)^{\circ}=\CG^{(0)}$.  

For the definition of a groupoid $C^*$-algebra we refer the reader to \cite{S}.


\subsection{A groupoid associated to $\Phi=(\{U_t\}_{t \in \mathbb{F}}, \{\varphi_t\}_{t \in \mathbb{F}})$}

We let
$$\mathbb{F} \ltimes_\varphi \partial E=\{(t, \mu) \in \mathbb{F} \times \partial E : \mu \in  U_t\}$$
denote the transformation groupoid associated with the partial action $ \Phi=(\{U_t\}_{t \in \mathbb{F}}, \{\varphi_t\}_{t \in \mathbb{F}})$, where the composition and the inversion are given by 
$$(t, \mu)(s, \nu)=(ts, \mu) ~\text{if}~\varphi_{t^{-1}}(\mu)=\nu,$$
$$(t, \mu)^{-1}=(t^{-1}, \varphi_{t^{-1}}(\mu))$$
and the range map and the source map are given by
$$r(t,\mu)=\mu ~\text{and }~s(t, \mu)=\varphi_{t^{-1}}(\mu),$$ where we identify 
 $(\mathbb{F} \ltimes_\varphi \partial E)^{(0)} $  with $\partial E$. Then, $\mathbb{F} \ltimes_\varphi \partial E$ is a totally disconnected ample Hausdorff groupoid (see \cite[Lemma 4.3 and Theorem 4.4]{CasK2} and \cite[Corollary 5.10]{CasK1}).  
 Since there is an isomorphism $\rho: C^*(\mathbb{F} \ltimes_\varphi \partial E) \to C_0(\partial E) \rtimes_{\hat{\varphi}} \mathbb{F}$ given by 
 $$\rho(f)=\sum_{t \in \mathbb{F}} f_t\dt_t$$
for $f \in C_c(\mathbb{F} \ltimes_\varphi \partial E)$, where $f_t: U_t \to \mathbb{C}$ is given by $f_t(x)=f(t,x)$ (see \cite[Proposition 3.1]{A}), one can deduce that there is an isomorphism  $\kp : C^*(\CB,\CL,\theta,\CI_\af)  \rightarrow C^*(\mathbb{F} \ltimes_\varphi \partial E)$  given on the generators of $C^*(\CB, \CL,\theta, \CI_\af)$ by 
\begin{align*} 
\kp(p_A)&=1_{\{\emptyset\} \times \CN(\emptyset,A)} \\
\kp(s_{\af,B})&=1_{\{\af\}\times \CN(\af,B)}.
\end{align*}
It is then straightforward to check that $$\kp(s_{\af,A}s_{\bt,A}^*)=1_{ \{\af\bt^{-1}\} \times \CN(\af,A)}$$ for $\af,\bt \in \CL^*$ and $A \in \CI_\af \cap \CI_\bt$.

\begin{remark} Let $\mathbb{F} \ltimes_\varphi \partial E$ be the transformation groupoid defined above. If $\CB$ and $\CL$ are countable, then $\mathbb{F} \ltimes_\varphi \partial E$ is second countable. 
\end{remark}

\section{The diagonal subalgebra of $C^*(\CB, \CL,\theta, \CI_\af)$}\label{The diagonal subalgebra}

Let $(\CB, \CL, \theta, \CI_\af)$ be a generalized Boolean dynamical system and $D$ be the diagonal subalgebra of $C^*(\CB, \CL,\theta, \CI_\af):=C^*(p_A, s_{\af,B})$. 
In this section,  we prove that the spectrum $\widehat{D}$ of $D$ is homeomorphic to $\partial E$. Precisely, we shall construct a homeomorphism $\Phi$ from $\partial E$ to $\widehat{D}$ such that 
$$\Phi(\mu)(s_{\af,A}s_{\af,A}^*)= 
\left\{
   \begin{array}{ll}
    1  & \hbox{if\ }  \mu \in \CN(\af,A),   \\
0 & \hbox{otherwise}\\
        \end{array}
\right. $$
 for $\mu \in \partial E$. 
 The non-trivial part is to define a well-defined continuous linear map $\Phi(\mu)$ satisfying the above equation  on $\operatorname{span}\{s_{\af,A}s_{\af,A}^*: \af \in \CL^* ~\text{and}~ A \in \CI_\af \}$. To achieve this, we first rewrite a finite linear combination of elements of the form $s_{\af,A}s_{\af,A}^*$ as a finite linear combination of orthogonal projection.

Before we begin, define 
$$E(S):=\{(\af, A, \af): \af \in \CL^* ~\text{and}~ \emptyset \neq A \in \CI_\af\}\cup \{0\}.$$
This represents the semilattice of idempotents of an inverse semigroup $S_{(\CB, \CL, \theta, \CI_\af)}:=\{(\af, A, \bt): \af,\bt \in \CL^* ~\text{and}~  A \in \CI_\af \cap \CI_\bt ~\text{with}~ A \neq \emptyset \} \cup \{0\}. $
 The natural order in the semilattice $E(S)$ is given as follows: for $\af,\bt \in \CL^*$, $A \in \CI_\af$ and $B \in \CI_\bt$, we have 
$$(\af, A, \af)\leq (\bt, B, \bt) ~\text{if and only if}~\af=\bt\af' ~\text{and}~A \subseteq\theta_{\af'}(B)$$ (\cite[Lemma 3.1]{CasK1}).

  The following lemma is inspired by \cite[Lemma 3.3]{SW} and \cite[Lemma 6.2]{BCM2}.

\begin{lem}\label{mo projections} Let $F \in E(S)\setminus \{0\}$ be a finite set such that for all $u,v \in F$, $uv=0$, $u \leq v$ or $v \leq u$. For $u=(\af,A,\af) \in F$, define
$$q^F_u:= s_{\af,A}s_{\af,A}^* \prod_{(\bt,B,\bt) \in F,\ (\bt,B,\bt) < (\af,A, \af)} (s_{\af,A}s_{\af,A}^* - s_{\bt,B}s_{\bt,B}^*).$$
Then, the  $q_u^F$  are mutually orthogonal projections in $\operatorname{span}\{s_{\bt,B}s_{\bt,B}^*: (\bt,B,\bt) \in F \}$. Also, for each $(\af,A,\af) \in F$,  we have 
\begin{align}\label{sum of orthogonal projections} s_{\af,A}s_{\af,A}^*=\sum_{u \in F, ~~ u \leq (\af,A,\af)} q_u^F.
\end{align}
\end{lem}

\begin{proof} Note that for each $u=(\af,A,\af) \in F$, $\{s_{\bt,B}s_{\bt,B}^*: (\bt, B, \bt) \in F ~\text{and}~ (\bt, B, \bt) < (\af,A,\af)\}$ is a finite set of commuting subprojections of $s_{\af,A}s_{\af,A}^*$. So, $q_u^F$ is a product of commuting projections, and hence, it is a non-zero projection in $\operatorname{span}\{s_{\bt,B}s_{\bt,B}^*: (\bt,B,\bt) \in F \}$.
 Let $u=(\af,A,\af), v=(\bt,B,\bt) \in F$ be such that $u \neq v$. 
If $uv=0$, then $s_{\af,A}s_{\af,A}^* s_{\bt,B}s_{\bt,B}^*=0$, so we have $q^F_u q^F_v=0$.
If $v < u$, then $s_{\af,A}s_{\af,A}^* -s_{\bt,B}s_{\bt,B}^* $ is  a factor of $q_u^F$ and $s_{\bt,B}s_{\bt,B}^*$ is a factor of $q_v^F$. Since 
$s_{\bt,B}s_{\bt,B}^* < s_{\af,A}s_{\af,A}^*$, we have $q^F_u q^F_v=0$.
By the same argument, we have $q^F_u q^F_v=0$ if $u < v$.

We prove (\ref{sum of orthogonal projections}) by induction on $|F|$. 
If $|F|=1$, the result is trivial. 
Let $n >1$ and suppose that the result holds true whenever $|F| < n$. 
Let $F \subseteq E(S) \setminus \{0\}$  with  $|F| = n$ satisfy the hypothesis of the lemma. Choose a minimal element $(\gm,C,\gm) \in F$ and define $G = F \setminus \{(\gm,C, \gm)\}$. 
Since $(\gm,C,\gm)$ is minimal in $F$ then
 $$\sum_{u \in F,~u \leq (\gm,C,\gm)} q_u^F=q_{ (\gm,C,\gm)}^F=s_{\gm,C}s_{\gm,C}^*.$$
 So, (\ref{sum of orthogonal projections}) holds for $(\gm,C,\gm)$. 
To show that (\ref{sum of orthogonal projections}) is true for each $(\af,A,\af) \in G$, we first claim that, for a given $u \in G$, 
\begin{eqnarray} \label{F into G}
q^F_{u} =\left\{           
  \begin{array}{ll}
    q^G_{u},  & \hbox{if\ }  u(\gm,C,\gm)=0,   \\
q^G_{u} -q^G_{u}s_{\gm,C}s_{\gm,C}^* & \hbox{if\ } (\gm,C,\gm) < u. 
   \end{array}
   \right. 
\end{eqnarray} 
If $u(\gm,C,\gm)=0$, then we clearly have that $q^F_{u}=q^G_{u}$. 
Let $u=(\bt,B,\bt)$ and assume that $(\gm,C,\gm) < u$.
Then, we have that
\begin{align*} q^F_{u} &=s_{\bt,B}s_{\bt,B}^*(s_{\bt,B}s_{\bt,B}^*-s_{\gm,C}s_{\gm,C}^*) \prod_{(\dt,D,\dt) \in G,~  (\dt,D,\dt) < u} (s_{\bt,B}s_{\bt,B}^* - s_{\dt,D}s_{\dt, D}^*) \\
&=s_{\bt,B}s_{\bt,B}^*(s_{\bt,B}s_{\bt,B}^*-s_{\gm,C}s_{\gm,C}^*) q^G_{u} \\ 
&=q^G_{u}  -q^G_{u} s_{\gm,C}s_{\gm,C}^*.
\end{align*}
Now, Let  $(\af,A,\af) \in G$. 
If $(\af,A,\af)(\gm,C,\gm)=0$, then by the equation (\ref{F into G}) and the induction hypothesis, we have 
$$\sum_{u \in F, ~u \leq (\af,A,\af)}q^F_u=\sum_{u \in G,~ u \leq (\af,A,\af)}q^G_u=s_{\af,A}s_{\af,A}^*.$$
If   $(\gm,C,\gm) < (\af,A,\af)$, then 
\begin{align*} \sum_{u \in F, ~ u \leq (\af,A,\af)}q^F_u
 &=q^F_{(\gm,C,\gm)} + \sum_{u \in G,~u \leq (\af,A,\af)}q^F_u \\
 &=s_{\gm,C}s_{\gm,C}^* +  \sum_{u \in G,~ u \leq (\af,A,\af)}q^F_u \\
&  = s_{\gm,C}s_{\gm,C}^* + \sum_{u \in G,~ u \leq (\af,A,\af)} (q_u^G - q_u^G s_{\gm,C}s_{\gm,C}^* ) \\
&=s_{\gm,C}s_{\gm,C}^* + s_{\af,A}s_{\af,A}^* - s_{\af,A}s_{\af,A}^* s_{\gm,C}s_{\gm,C}^* \\
&= s_{\af,A}s_{\af,A}^*,
\end{align*}
where the third equality follows from the equation (\ref{F into G}) since $q_u^F= q_u^G - q_u^G s_{\gm,C}s_{\gm,C}^*$ even when $u(\gm,C,\gm)=0$ and the last equality follows from the induction hypothesis. Thus, we are done. 
\end{proof}

\begin{lem}\label{disjoint}(cf.\cite[Lemma 6.3]{BCM2}) Let $F \subseteq E(S) \setminus \{0\}$ be a finite set. Then there exists $F' \subseteq E(S) \setminus \{0\}$ such that $F'$ satisfies the hypothesis of Lemma \ref{mo projections} and the following conditions:
\begin{enumerate}
\item for all $(\af,A,\af) \in F$, there exist $(\af,A_1,\af), \cdots, (\af, A_n, \af) \in F'$ such that $A$ is the  union of $A_1, \cdots, A_n$;
\item the word that appear in elements of $F'$ are the same as those that appear in elements of $F$;
\item if $(\af,A,\af), (\af,B,\af) \in F'$ and $A \neq B$, then $A \cap B =\emptyset$.
\end{enumerate}
\end{lem}

\begin{proof} For a given finite set $F \subseteq E(S) \setminus \{0\}$, define $m=\max \{|\af|: (\af,A,\af) \in F\}$. We proceed by induction on $m$. If $m=0$, then $F=\{(\emptyset, B_1, \emptyset), \cdots, (\emptyset, B_l, \emptyset)\}$. Define 
$$\CC=\Big\{ \bigcap_{i \in I_1} B_i \setminus \bigcup_{j \in I_2} B_j : I_1 \cup I_2 =\{1, \cdots, l\}, I_1 \cap I_2 =\emptyset \Big\} \setminus \{\emptyset\}$$
and $F'=\{(\emptyset, B ,\emptyset): B \in \CC \}$. Then, $F'$ is the desired one. 

For $m > 0$, suppose that the result holds whenever $G \subset E(S)$ is a finite set with $\max \{|\af|: (\af,A, \af) \in G\} < m$. Fix a finite set $F \subseteq E(S) \setminus \{0\}$ with $\max \{|\af|: (\af,A,\af) \in F\}=m$. Write $F=G_1 \cup G_2$ where $G_1=\{(\af,A,\af) \in F: |\af|< m\}$ and $G_2=\{(\af,A,\af): |\af|=m\}$. Then by the induction hypothesis, there exists $G_1'$ associated to $G_1$ that satisfies the conditions in the statement. Let $L_2=\{\af \in \CL^*:  (\af,A,\af) \in G_2 ~\text{for some}~ A \in \CI_\af \}$.  For each $\af \in L_2$,  we put $$\CD_\af=\{\theta_{\af''}(A): (\af',A, \af') \in G_1' \cup G_2 ~\text{and}~ \af=\af'\af''\}$$
and consider $\CC_\af$ constructed from $\CD_\af$ as $\CC$ was constructed from $\{B_1, \cdots, B_l\}$ in the case $m=0$.
Now, define $F_2'= \bigcup_{\af \in L_2}\{(\af,B, \af) \in E(S): B \in \CC_\af \}$ and $F'=G_1' \cup F_2'$.

Clearly, $0 \notin F'$. Let $u=(\af,A,\af) , v=(\bt,B,\bt) \in F'$.
If $u$ and $v$ are both elements of either $G_1'$ or both elements of  $F_2'$, then $u$ and $v$ are such that $uv=0, u \leq v$ or $v \leq u$.
Next, suppose that  $u \in F_2'$ and $v \in G_1'$ so that $|\bt| < |\af|$. If $\af$ and $\bt$ are not comparable, then $uv=0$. Otherwise, $\af=\bt\af'$ and, by the construction of $\CC_\af$, $A \subseteq \theta_{\af'}(B)$ or $A \cap \theta_{\af'}(B) =\emptyset$. In this case $u \leq v$ or $uv=0$. From the construction, it is easy to see that the set $F'$ satisfies the conditions (1),(2) and (3) in the statement. 
\end{proof}

\begin{remark}Let $\mu=(e^{\mu_i}_{\eta_i})_{i=1}^{|\mu|} \in \partial E$ be given  and a finite set $F \subseteq E(S) \setminus \{0\}$ be such that for all $u,v \in F$, $ uv=0, u \leq v$ or $v \leq u$. Then, the set $\{(\af,A,\af) \in F: \mu \in \CN(\af,A)\} $ is totally ordered if it is non-empty.
\end{remark}

\begin{proof} Choose $ u:=(\mu_1 \cdots \mu_l, A, \mu_1 \cdots \mu_l), v:= (\mu_1 \cdots \mu_n, B, \mu_1 \cdots \mu_n) \in \{(\af,A,\af) \in F: \mu \in \CN(\af,A)\}$ with $l \leq n$. Since $d(e^{\mu_{l}}_{\eta_l})=
r(e^{\mu_{l+1}}_{\eta_{l+1}} \cdots e^{\mu_n}_{\eta_n})=f_{\emptyset[\mu_{l+1} \cdots \mu_n]}(\eta_n \cap \CI_{\mu_{l+1} \cdots \mu_n})$, it follows that $\theta_{\mu_{l+1} \cdots \mu_n}(A), B \in \eta_n$, and hence,  $\theta_{\mu_{l+1} \cdots \mu_n}(A)\cap B  \neq \emptyset$. Thus, $uv=(\mu_1 \cdots \mu_n, \theta_{\mu_{l+1} \cdots \mu_n}(A)\cap B, \mu_1 \cdots \mu_n) \neq 0$. 
\end{proof}

\begin{lem} Let $\mu=(e^{\mu_i}_{\eta_i})_{i=1}^{|\mu|} \in \partial E$ and a finite set $F \subseteq E(S) \setminus \{0\}$ be such that $\{(\af,A,\af) \in F: \mu \in \CN(\af,A)\} \neq \emptyset$ and for all $u,v \in F$, $ uv=0, u \leq v$ or $v \leq u$. If we let  $w =\min\{(\af,A,\af) \in F: \mu \in \CN(\af,A)\}$, then the projection $q_w^F$ is non-zero. 
 \end{lem}

\begin{proof} We first claim that there exists  $z \in E(S)\setminus \{0\}$  such that $z \leq w$ and $zu=0$ for all $u \in F$ with $u < w$.
The result is trivial if there is no $u \in F$ with $u < w$. So, suppose that there exists at least one such $u$. Let $w=(\mu_1 \cdots \mu_l, A, \mu_1 \cdots \mu_l)$ for some $l \geq 0$ and $A \in \CI_{\mu_1 \cdots \mu_l}$. 
First, consider the case that $|\mu|=\infty$. Let $n = \max\{|\bt|: (\bt, B, \bt) \in F ~\text{such that}~ (\bt, B, \bt) < w\}$. Choose an element $u \in F$ with $u < w$ of the form $u=(\mu_1 \cdots \mu_m, B , \mu_1 \cdots \mu_m)$. Since $w=\min\{(\af,A,\af) \in F: \mu \in \CN(\af,A)\}$, we have $\mu \notin \CN(\mu_1 \cdots \mu_m, B )$. So, $B \notin \eta_{m}$, and hence, $\theta_{\mu_{m+1} \cdots \mu_n}(B) \notin \eta_n$. Since $\eta_n$ is an ultrafilter, there exists $C_u \in \eta_n$ such that $C_u \cap \theta_{\mu_{m+1} \cdots \mu_n}(B) =\emptyset$. Let $C =\cap _{u \in F, u < w}C_u \in \eta_n$ and define $z=(\mu_1 \cdots \mu_n, C \cap \theta_{\mu_{l+1} \cdots \mu_n}(A), \mu_1 \cdots \mu_n)$. Then $z \neq 0$ since $C \cap \theta_{\mu_{l+1} \cdots \mu_n}(A) \in \eta_n$ and $z \leq w$. Also, it is easy to see that $zu=0$ for all $u \in F$ with $u < w$.

For the second, suppose that $0 \leq |\mu| < \infty$. By the same argument as the above case, one can choose $C \in \eta_{|\mu|}$ such that for all $u \in F$ with $u < w$ of the form $(\mu_1 \cdots \mu_m, B, \mu_1 \cdots \mu_m)$, we have that $C \cap \theta_{\mu_{m+1} \cdots \mu_{|\mu|}}(B) = \emptyset$. Put $D=C \cap \theta_{\mu_{l+1} \cdots \mu_{|\mu|}}(A) \in \eta_{|\mu|}$. Since $d(\mu) \in E^0_{sg}$, we have $D \notin \CB_{reg}$ by \cite[Lemma 7.9]{CasK1}. Then there exists $\emptyset \neq D' \subseteq D $ such that either $|\Delta_{D'}|=0$ or $|\Delta_{D'}|=\infty$. 
If  $|\Delta_{D'}|=\infty$, then choose $\dt \in \Delta_{D'}$ be the letter that is different from $\bt_{|\mu|+1}$ for all $\bt$ such that $|\bt| \geq |\mu|+1$ and that is appears in an element $v=(\bt,B, \bt) \in F$. Then, $z=(\mu \dt, \theta_\dt(D'), \mu \dt)$ is the desired one. 
If $|\Delta_{D'}|=0$, define $z=(\CP(\mu), D', \CP(\mu))$ so that $z \neq 0$ and $z \leq w$.
Let $u =(\bt,B, \bt) \in F$ be such that $ u < w$. If $\bt$ is not comparable with $\CP(\mu)$, then it is clear that $zu=0$. If $\bt$ is a beginning of $\CP(\mu)$, then $zu=0$. Finally, if $\bt=\CP(\mu)\gm$ for some $\gm \in \CL^{\geq 1}$, then $zu=0$ because $\theta_{\gm}(D')=\emptyset$.

Now, we say $w=(\af,A,\af)$ and $z=(\gm,C, \gm)$. Then, we see that 
\begin{align*} & s_{\gm,C}s_{\gm,C}^*q^F_w \\
 &=s_{\gm,C}s_{\gm,C}^* \Big( s_{\af,A}s_{\af,A}^* \prod_{(\bt,B,\bt) \in F,\ (\bt,B,\bt) < (\af,A, \af)} (s_{\af,A}s_{\af,A}^* - s_{\bt,B}s_{\bt,B}^*) \Big) \\
&=s_{\gm,C}s_{\gm,C}^*.
\end{align*}
Since $s_{\gm,C}s_{\gm,C}^* \neq 0$, we have $q^F_w \neq 0$.
\end{proof}

We are now ready to prove our main result of this section. It is a generalization of \cite[Theorem 3.7]{Web} and \cite[Theorem 6.9]{BCM2}.   

\begin{thm} Let $(\CB, \CL,\theta, \CI_\af)$ be a generalized Boolean dynamical system. Then, for each $\mu \in \partial E$, there exists a unique $\Phi(\mu) \in \widehat{D}$ such that 
$$\Phi(\mu)(s_{\af,A}s_{\af,A}^*)= 
\left\{
   \begin{array}{ll}
    1  & \hbox{if\ }  \mu \in \CN(\af,A),   \\
0 & \hbox{otherwise. }
        \end{array}
\right. $$
Moreover, the map $\Phi: \partial E \to \widehat{D}$, $\mu \mapsto \Phi(\mu)$ is a homeomorphism.
\end{thm}

\begin{proof} Fix $\mu=(e^{\mu_i}_{\eta_i})_{i=1}^{|\mu|} \in \partial E$ and $\sum_{(\af,A, \af) \in F} \ld_{(\af,A,\af)} s_{\af,A}s_{\af,A}^* \in \text{span}\{s_{\af,A}s_{\af,A}^*: (\af,A, \af) \in E(S)\}$ for some finite $F \subset E(S)$.
 We first claim that 
 
 \begin{align} \Big\|  \sum_{(\af,A, \af) \in F} \ld_{(\af,A,\af)} s_{\af,A}s_{\af,A}^* \Big\| \geq  \Big|\sum_{(\af,A,\af) \in F; \mu \in \CN(\af,A)} \ld_{(\af,A, \af)} \Big|.\end{align}

If $ \mu \notin \CN(\af, A)$ for all $(\af,A, \af) \in F$, then the result is clear. So, suppose $\mu \in \CN(\af,A)$ for some $(\af,A, \af) \in F$. Let $F'$ be the set  constructed from $F$ as in Lemma \ref{disjoint}. Note that $\{(\gm, C, \gm) \in F': \mu \in \CN(\gm,C)\} \neq \emptyset$. In fact, for $(\af,A,\af) \in F$ such that $\mu \in \CN(\af,A)$, there exist $(\af,B_1,\af)$, $\cdots$, $(\af, B_m, \af) \in F'$ such that $A =\sqcup_{i=1}^m B_i$. Since $\eta_{|\af|}$ is prime  and $A \in \eta_{|\af|}$, $B_{i_0} \in \eta_{|\af|}$ for some $i_0 \in \{1, \cdots,m\}$. So, $\mu \in \CN(\af, B_{i_0})$ for some $(\af, B_{i_0}, \af) \in F'$.

We also note that for  each $(\gm, C, \gm) \in F'$ and $(\af,A, \af) \in F$ with $(\gm,C,\gm) \leq (\af,A,\af)$, if $(\af,B_i,\af), (\af, B_j,\af) \in F'$ with  $ (\gm,C,\gm) \leq (\af,B_i,\af) \leq (\af,A,\af)$ and   $ (\gm,C,\gm) \leq (\af,B_j,\af) \leq (\af,A,\af)$, then $B_i=B_j$. In deed, saying $\gm=\af\gm'$, we see that $C \subseteq \theta_{\af'}(B_i) \cap \theta_{\af'}(B_j)=\theta_{\af'}(B_i \cap B_j )$. If $B_i \neq B_j$, then $B_i \cap B_j =\emptyset$, and hence, it follows that  $C =\emptyset$, this is not the case. So, $B_i=B_j$.  

Then, we have
\begin{align*} &\sum_{(\af,A, \af) \in F} \ld_{(\af,A,\af)} s_{\af,A}s_{\af,A}^*  \\
&= \sum_{(\af,A, \af) \in F} \ld_{(\af,A,\af)} s_{\af,A}\Big( \sum_{(\af,B,\af) \in F', B \subseteq A}p_B\Big)s_{\af,A}^* \\
&= \sum_{(\af,A, \af) \in F} \ld_{(\af,A,\af)}\Big( \sum_{(\af,B,\af) \in F', B \subseteq A}  s_{\af,A} p_B s_{\af,A}^*\Big) \\
&= \sum_{(\af,A, \af) \in F} \ld_{(\af,A,\af)}\sum_{(\af,B,\af) \in F', B \subseteq A}  s_{\af,B}  s_{\af,B}^* \\
&=\sum_{(\af,A, \af) \in F} \ld_{(\af,A,\af)}\sum_{(\af,B,\af) \in F', B \subseteq A}   \Big( \sum_{(\gm,C, \gm) \in F', (\gm,C,\gm) \leq (\af,B,\af)} q^{F'}_{(\gm,C,\gm)}\Big) \\
&=\sum_{(\af,A, \af) \in F} \ld_{(\af,A,\af)} \Big( \sum_{(\gm,C, \gm) \in F', (\gm,C,\gm) \leq (\af,A,\af)} q^{F'}_{(\gm,C,\gm)} \Big)\\
&=\sum_{(\gm,C, \gm) \in F'} \Big(\sum_{(\af,A,\af) \in F, (\gm,C,\gm) \leq (\af,A,\af)}   \ld_{(\af,A,\af)}\Big) q^{F'}_{(\gm,C,\gm)},
\end{align*}
where  the fifth equation follows from the note in the previous paragraph.

Now, let $\omega=\min \{(\gm, C,\gm) \in F': \mu \in \CN(\gm,C)\}$. We observe that if $w:=(\gm,C,\gm)$ with $w \leq (\af,A,\af)$, then $w \in \CN(\af,A)$. In deed, saying $\gm=\af\gm'$, we have $C \subseteq \theta_{\gm'}(A)$. Since $C \in \eta_{|\gm|}$, we have   $\theta_{\gm'}(A) \in \eta_{|\gm|}$, and hence,   $A \in \eta_{|\af|}$. So, $w \in \CN(\af,A)$.
 Thus,  it follows that  
\begin{align*} 
& \Big\| \sum_{(\af,A, \af) \in F} \ld_{(\af,A,\af)} s_{\af,A}s_{\af,A}^*  \Big\| \\
&= \Big\| \sum_{(\gm,C, \gm) \in F'} \Big(\sum_{(\af,A,\af) \in F, (\gm,C,\gm) \leq (\af,A,\af)}   \ld_{(\af,A,\af)}\Big) q^{F'}_{(\gm,C,\gm)}\Big\|\\
&= \max_{(\gm,C,\gm) \in F', q^{F'}_{(\gm,C,\gm)} \neq 0} \Big| \sum_{(\af,A,\af) \in F, (\gm,C,\gm) \leq (\af,A,\af)}  \ld_{(\af,A,\af)} \Big| \\
& \geq  \Big| \sum_{(\af,A,\af) \in F, \omega \leq (\af,A,\af)}   \ld_{(\af,A,\af)} \Big| \\
&= \Big| \sum_{(\af,A,\af) \in F, \omega \in \CN(\af,A)}   \ld_{(\af,A,\af)} \Big|.
\end{align*} 
Hence,  the formula 
$$\Phi(\mu)\Big(\sum_{(\af,A, \af) \in F} \ld_{(\af,A,\af)} s_{\af,A}s_{\af,A}^* \Big)=\sum_{(\af,A, \af) \in F, \ \mu \in \CN(\af,A)} \ld_{(\af,A,\af)} $$ determines a well-defined norm-decreasing  linear map $\Phi(\mu) $ from $ \text{span}\{s_{\af,A}s_{\af,A}^*: (\af,A, \af) \in E(S)\}$ to $\mathbb{C}$.
To see that $\Phi(\mu)$ is a homomorphism, it is enough to show that 
\begin{align}\label{homo}
\Phi(\mu)(s_{\af,A}s_{\af,A}^*s_{\bt,B}s_{\bt,B}^*)=\Phi(\mu)(s_{\af,A}s_{\af,A}^*)\Phi(\mu)(s_{\bt,B}s_{\bt,B}^*).
\end{align}
We observe that 
\[
\Phi(\mu)(s_{\af,A}s_{\af,A}^*s_{\bt,B}s_{\bt,B}^*)= \left\{ 
\begin{array}{ll}
        1  & \hbox{if\ }\af=\bt\af' ~\text{and}~ \mu \in \CN(\af, A \cap \theta_{\af'}(B)) \\
    & \hbox{or\ } \bt=\af\bt' ~\text{and}~ \mu \in \CN(\bt, B \cap \theta_{\bt'}(A)) \\
    0   & \hbox{otherwise} \\
\end{array}
\right.
\]
and that 
\[
\Phi(\mu)(s_{\af,A}s_{\af,A}^*)\Phi(\mu)(s_{\bt,B}s_{\bt,B}^*)= \left\{ 
\begin{array}{ll}
        1  & \hbox{if\ }  \mu \in \CN(\af,A) \cap \CN(\bt,B)  \\
    0   & \hbox{otherwise.} \\
\end{array}
\right.
\]
Then by \cite[Lemma 3.9]{CasK2}, the equation (\ref{homo}) follows.

Thus,  $\Phi(\mu)$  is a nonzero continuous homomorphism on a dense subspace of $D$, and hence,  it extends uniquely to a nonzero  homomorphism $\Phi(\mu): D \to \mathbb{C}$. 

To show that $\Phi$ is a homeomorphism, we prove the following:

  \begin{itemize}
    \item  The map $\Phi: \partial E \to \widehat{D}$ is  surjective; 
\end{itemize}

 fix $\phi \in \widehat{D}$ and let $\xi=\{(\af,A,\af) \in E(S): \phi(s_{\af,A}s_{\af,A}^*)=1\}$. Then $\xi$ is a filter; clearly, $0 \notin \xi$. If $(\af, A, \af) \leq (\bt, A, \bt)$ and $\phi(s_{\af,A}s_{\af,A}^*)=1$. Then, 
$$\phi(s_{\bt,B}s_{\bt,B}^*)=\phi(s_{\af,A}s_{\af,A}^*) \phi(s_{\bt,B}s_{\bt,B}^*)=\phi(s_{\af,A}s_{\af,A}^*s_{\bt,B}s_{\bt,B}^*)=\phi(s_{\af,A}s_{\af,A}^*)=1.$$
Thus, $(\bt,B,\bt) \in \xi$. It is straightforward to check that if $(\af, A, \af), (\bt, A, \bt) \in \xi$, then $(\af, A, \af)(\bt, A, \bt) \in \xi $.

Then, by \cite[Lemma 3.5]{CasK1}, there exists a unique word $\bt \in \CW^{\leq \infty}$ associated with $\xi$ such that for all $(\af,A,\af) \in \xi$,  $\af$ is an initial segment of $\bt$. For each $0 \leq n \leq |\bt|$, put 
$\xi_n:=\{A \in \CI_{\bt_{1,n}}: \phi(s_{\bt_{1,n},A}s_{\bt_{1,n},A}^*)=1\}$. 
Then by \cite[Proposition 3.7]{CasK1}, $\{\xi_n\}_{0 \leq n \leq |\bt|}$ is a complete family for $\bt$.

 Now we put $\eta_1=\xi_1$ and $\eta_n=h_{[\bt_{1,n-1}]\bt_n}(\xi_n)$ for $2\leq n\leq|\bt|$.
 Define 
\[\mu:=\begin{cases}
\xi_0 & \text{if } \bt=\emptyset, \\
(e^{\bt_n}_{\eta_n})_{n=1}^{|\bt|} &\text{if }|\bt|\geq 1.
\end{cases}\]
We claim that $\mu \in \partial E$. For the case that  $0 \leq |\bt| < \infty$, we first  show that $B \notin \CB_{reg}$ for all $B \in d(e^{\bt_{|\bt|}}_{\eta_{|\bt|}}) (=h_{[\bt]\emptyset}(\xi_{|\bt|})= \uparrow_{\CB} \xi_{|\bt|})$. Assume to the contrary that $B \in \CB_{reg}$ for some $B \in \uparrow_{\CB} \xi_{|\bt|}$. One then  can choose 
$A \in \xi_{|\bt|} \cap \CB_{reg}$, and hence, we have $p_A =\sum_{\gm \in \Delta_A} s_{\gm, \theta_\gm(A)}s_{\gm, \theta_\gm(A)}^*$. Since $A \in \xi_{|\bt|}$, we see that 
$$1=\phi(s_{\bt,A}s_{\bt,A}^*)=\phi(s_{\bt,A}p_As_{\bt,A}^*)=\sum_{\gm \in \Delta_A} \phi(s_{\bt\gm, \theta_\gm(A)}s_{\bt\gm, \theta_\gm(A)}^*).$$
Thus, $(\bt\gm, \theta_\gm(A), \bt\gm) \in \xi$  for some $\gm \in \Delta_A$, but this contradicts to the fact that $\bt$ is the largest word for $\xi$. 
Hence,  $B \notin \CB_{reg}$ for all $B \in d(e^{\bt_{|\bt|}}_{\eta_{|\bt|}}) $. It then follows by \cite[Lemma 7.9]{CasK1} that $d(e^{\bt_{|\bt|}}_{\eta_{|\bt|}}) \in E^0_{sg}$. 

Next, we see that 
\begin{align*}
    d(e^{\bt_n}_{\eta_n})  &=h_{[\bt_{n}]\emptyset}(\eta_n) \\
    &=h_{[\bt_{1,n}]\emptyset}(\xi_n)\\   
    &=h_{[\bt_{1,n}]\emptyset}(f_{\bt_{1,n}[\bt_{n+1}]}(\xi_{n+1}))\\
    &=f_{\emptyset[\bt_{n+1}]}(h_{[\bt_{1,n}]\bt_{n+1}}(\xi_{n+1}))\\
    &=f_{\emptyset[\bt_{n+1}]}(\eta_{n+1}) \\
    &=r(e^{\bt_{n+1}}_{\eta_{n+1}})
\end{align*}
for $1 \leq n < |\bt|$. Therefore, $\mu \in \partial E$ if  $0 \leq |\bt| < \infty$.
The same argument as above  says that $\mu \in  \partial E$ if  $|\bt|=\infty$.

To see that $\Phi(\mu)=\phi$, notice that For $(\af,A,\af) \in E(S)$, we have 
\begin{align*} \phi(s_{\af,A}s_{\af,A}^*)=1 & \iff \af=\bt_{1,n}~\text{and}~ A \in \xi_n   ~\text{for some}~ n \geq 1 \\
& \iff \Phi(\mu)(s_{\af,A}s_{\af,A}^*)=1.
\end{align*}
    Since both $\phi(s_{\af,A}s_{\af,A}^*)$ and $\Phi(\mu)(s_{\af,A}s_{\af,A}^*)$ take values in $\{0,1\}$, it follows that $\Phi(\mu)=\phi$.
    
       \begin{itemize}
    \item  The map $\Phi$ is injective;
\end{itemize}

    suppose that $\Phi(\mu)=\Phi(\nu)$ for $\mu=(e^{\mu_i}_{\eta_i})$ and $ \nu=(e^{\nu_i}_{\xi_i}) \in \partial E$.    
    For each $n \in \mathbb{N}$, let $n \wedge |\mu|=\min \{n , |\mu|\}$ for $\mu \in \partial E$. Fix $n \wedge |\mu| \in \mathbb{N}$ and choose $A \in \eta_{n \wedge |\mu|}$. Then, we have 
    $$\Phi(\nu)(s_{\CP(\mu)_{1, n \wedge |\mu|},A}s_{\CP(\mu)_{1, n \wedge |\mu|},A}^*)=\Phi(\mu)(s_{\CP(\mu)_{1, n \wedge |\mu|},A}s_{\CP(\mu)_{1, n \wedge |\mu|},A}^*)=1.$$
       So, $\CP(\nu)_{1, n \wedge |\mu|}=\CP(\mu)_{1, n \wedge |\mu|}$ and $A \in \xi_{n \wedge |\mu|}$ for all $n \in \mathbb{N}$. By symmetry, if $A \in \xi_{n \wedge |\mu|} $, we also have that 
    $\CP(\mu)_{1, n \wedge |\nu|}=\CP(\nu)_{1, n \wedge |\nu|}$ and $A \in \eta_{n \wedge |\mu|}$ for all $n \in \mathbb{N}$. Thus, $|\mu|=|\nu|$, $\CP(\mu)_{1, n}=\CP(\nu)_{1, n }$ and $\eta_n=\xi_n$ for all $n \leq |\mu|$. Thus, $\mu=\nu$, and hence, $\Phi$ is injective.

    \begin{itemize}
    \item  The map $\Phi$ is continuous;
\end{itemize}

suppose that $\{\mu^{\ld}\}_{\ld \in \Lambda}$ converges to $\mu$, where $\mu^\ld=(e^{\mu_i^{(\ld)}}_{\eta_i^{(\ld)}})_{i=1}^{|\mu^\ld|}$ and $\mu=(e^{\mu_i}_{\eta_i})_{i=1}^{|\mu|}$. We claim that $\{\Phi(\mu^\ld)\}_{\ld \in \Lambda} $ converges to $ \Phi(\mu)$.  
Let $\af:=\af_1 \cdots \af_n \in \CL^*$ and $A \in \CI_\af$. First suppose that $\Phi(\mu)(s_{\af,A}s_{\af,A}^*)=1$.
If $\af =\emptyset$, then $A \in r(\mu)$. Since $\{r(\mu^\ld)\}_{\ld \in \Lambda}$ converges to $r(\mu)$, there exists $\ld_0 \in \Lambda$ such that $A \in r(\mu^\ld)$ for all $\ld \geq \ld_0$. Thus, $\Phi(\mu^\ld)(s_{\af,A}s_{\af,A}^*)=1$ for all $\ld \geq \ld_0$.
If $\af \neq \emptyset$, then $\mu \in \CN(\af,A)$. 
Choose $B \in \CI_{\af_1}$ such that $A \subseteq \theta_{\af_{2,n}}(B)$. 
One then can see that  $e^{\mu_1}_{\eta_1}e^{\mu_2}_{\eta_2} \cdots e^{\mu_{n}}_{\eta_{n}} \in \big( Z^1(\af_1, B) \times Z^1(\af_2, \theta_{\af_2}(B)) \times \cdots \times Z^1(\af_{n-1}, \theta_{\af_2 \cdots \af_{n-1}}(B)) \times Z^1(\af_{n}, A) \Big) \cap E^{n}$. Since $\mu^\ld \to \mu$, there is $\ld_0 \in \Lambda$ such that 
$e^{\mu_1^{(\ld)}}_{\eta_1^{(\ld)}}e^{\mu_2^{(\ld)}}_{\eta_2^{(\ld)}} \cdots e^{\mu_n^{(\ld)}}_{\eta_n^{(\ld)}}  \in \big( Z^1(\af_1, B) \times Z^1(\af_2, \theta_{\af_2}(B)) \times \cdots \times Z^1(\af_{n-1}, \theta_{\af_2 \cdots \af_{n-1}}(B)) \times Z^1(\af_{n}, A) \Big) \cap E^{n}$ for all $\ld \geq \ld_0$. Thus, $\Phi(\mu^\ld)(s_{\af,A}s_{\af,A}^*)=1$ for all $\ld \geq \ld_0$.

Secondly, suppose that  $\Phi(\mu)(s_{\af,A}s_{\af,A}^*)=0$. We consider the following  cases:

If $\af =\emptyset$, then $r(\mu)$ belongs to the open set $F^0 \setminus Z(A)$ for $F^0$. Since $\{r(\mu^\ld)\}_{\ld \in \Lambda}$ converges to $r(\mu)$, there exists $\ld_0 \in \Lambda$ such that $ r(\mu^\ld) \in F^0 \setminus Z(A)$ for all $\ld \geq \ld_0$. Thus, $\Phi(\mu^\ld)(s_{\af,A}s_{\af,A}^*)=0$ for all $\ld \geq \ld_0$.

If $1 \leq |\af| \leq |\mu|$ and $\af$ is not a beginning of $\CP(\mu)$, then we can find $\ld_0 \in \Lambda$ such that $\CP(\mu^\ld)_{1, |\af|}=\CP(\mu)_{1, |\af|} \neq \af$ for all $\ld \geq \ld_0$. Thus,  $\Phi(\mu^\ld)(s_{\af,A}s_{\af,A}^*)=0$ for all $\ld \geq \ld_0$.

If  $1 \leq |\af| \leq |\mu|$ and $\af$ is  a beginning of $\CP(\mu)$, then one can  find $\ld_1 \in \Lambda$ such that $\{e^{\mu_1^{(\ld)}}_{\eta_1^{(\ld)}}e^{\mu_2^{(\ld)}}_{\eta_2^{(\ld)}} \cdots e^{\mu_{|\af|}^{(\ld)}}_{\eta_{|\af|}^{(\ld)}} \}_{\ld \geq \ld_1}$ converges to $e^{\mu_1}_{\eta_1}e^{\mu_2}_{\eta_2} \cdots e^{\mu_{|\af|}}_{\eta_{|\af|}}$ in $E^{|\af|}$. So, $\CP(\mu^\ld)_{1, |\af|}=\CP(\mu)_{1, |\af|} = \af$ for all $\ld \geq \ld_1$. Moreover,  $\{\eta^{(\ld)}_{|\af|}\}_{\ld \geq \ld_1}$ converges to $\eta_{|\af|}$. 
Since $A \notin \eta_{|\af|} $, there exists $\ld_0 \geq \ld_1$ such that $A \notin \eta^{(\ld)}_{|\af|}$ for all $\ld \geq \ld_0$.  Thus,  $\Phi(\mu^\ld)(s_{\af,A}s_{\af,A}^*)=0$ for all $\ld \geq \ld_0$.

If $|\af| > |\mu|$ and $\CP(\mu)$ is not a beginning of $\af$, 
then we can find $\ld_0 \in \Lambda$ such that $\CP(\mu^\ld)_{1, |\mu|}=\CP(\mu) \neq  \af_{1, |\mu|}$ for all $\ld \geq \ld_1$. Thus,  $\Phi(\mu^\ld)(s_{\af,A}s_{\af,A}^*)=0$ for all $\ld \geq \ld_0$.

Lastly, we suppose that  $|\af| > |\mu|$ and $\CP(\mu)$ is a beginning of $\af$.  
Since $A \in \CI_\af$, there exists $B \in \CB$ such that $A \subseteq \theta_{\af_{2, |\af|}}(B)$. Then, $\theta_{\af_{2, |\mu|+1}}(B) \in \CI_{\af_{1,|\mu|+1}} \subseteq \CI_{\af_{|\mu|+1}}$ and observe that $K:=Z^1(\af_{|\mu|+1}, \theta_{\af_{2, |\mu|+1}}(B))$ is a compact subset of $E^1$. Then, there exists $\ld_0 \in \Lambda$
 such that for all $\ld \geq \ld_0$, either $|\mu^\ld|=|\mu|$ or $|\mu^\ld| > |\mu|$ and $\mu^\ld_{|\mu|+1} \notin K$. 
 In the first case, $\Phi(\mu^\ld)(s_{\af,A}s_{\af,A}^*)=0$ for all $\ld \geq \ld_0$ since $|\mu^\ld|=|\mu| < |\af|$ for all $\ld \geq \ld_0$.
In the second case, if $\af$ is not a beginning of $\mu^{\ld}$, then $\Phi(\mu^\ld)(s_{\af,A}s_{\af,A}^*)=0$ for all $\ld \geq \ld_0$. If $\af$ is a beginning of $\mu^{\ld}$, then $\theta_{\af_{2, |\mu|+1}}(B)  \notin \eta^{(\ld)}_{|\mu|+1}$ since $\mu^\ld_{|\mu|+1} \notin K$. Then $\theta_{|\mu|+2, |\af|}(\theta_{\af_{2, |\mu|+1}}(B) )=\theta_{2, |\af|}(B) \notin \eta^{(\ld)}_{|\af|}$. Since $A \subseteq \theta_{\af_{2, |\af|}}(B)$, we have $A \notin \eta^{(\ld)}_{|\af|}$. Thus, $\Phi(\mu^\ld)(s_{\af,A}s_{\af,A}^*)=0$ for all $\ld \geq \ld_0$.

Thus, $\Phi(\mu^\ld)(s_{\af,A}s_{\af,A}^*)$ converges to $\Phi(\mu)(s_{\af,A}s_{\af,A}^*)$ for all $\af \in \CL^*$ and $A \in \CI_\af$.  Then, a standard $\epsilon / 3$ argument shows that  $\Phi(\mu^\ld)(a)$ converges to $\Phi(\mu)(a)$ for every $a \in D$. Hence, $\Phi$ is continuous. 

  \begin{itemize}
    \item  $\Phi^{-1}$ is continuous;
\end{itemize}

 suppose that $\{\Phi(\mu^\ld)\}_{\ld \in \Lambda} $ converges to $ \Phi(\mu)$, where $\mu^\ld=(e^{\mu_i^{(\ld)}}_{\eta_i^{(\ld)}})_{i=1}^{|\mu^\ld|}$ and $\mu=(e^{\mu_i}_{\eta_i})_{i=1}^{|\mu|}$. We prove that $\{\mu^{\ld}\}_{\ld \in \Lambda}$ converges to $\mu$:
 
(1) Fix a basic open set $Z(A)$ containing $r(\mu) (=f_{\emptyset[\mu_1]}(\eta_1))$. Then, $\theta_{\mu_1}(A) \in \eta_1 $, so we have $\Phi(\mu)(s_{\mu_1,\theta_{\mu_1 }(A) }s_{\mu_1 ,\theta_{\mu_1}(A)}^*)=1$. Since $\Phi(\mu^\ld) \to \Phi(\mu)$, there is $\ld_0 \in \Lambda$ such that $\Phi(\mu^\ld)(s_{\mu_1,\theta_{\mu_1 }(A) }s_{\mu_1 ,\theta_{\mu_1}(A)}^*)=1$ for all $\ld \geq \ld_0$. Thus, $\mu_1^{(\ld)}=\mu_1$ and $\theta_{\mu_1 }(A) \in \eta_1^{(\ld)}$ for all $\ld \geq \ld_0$. Thus, $r(\mu^\ld) \in Z(A)$ for  $\ld \geq \ld_0$. Hence, $\{r(\mu^\ld)\}_{\ld}$ converges to $r(\mu)$. 

(2)  Let $1 \leq k \leq |\mu|$ with $k \neq \infty$.
Take  a basic open set $\big(Z^1(\mu_1, B_1) \times Z^1(\mu_2,B_2) \times \cdots \times Z^1(\mu_k, B_k)\big) \cap E^k$  in $E^k$ containing $e^{\mu_1}_{\eta_1}e^{\mu_2}_{\eta_2} \cdots e^{\mu_k}_{\eta_k}$.
Then, $\Phi(\mu)(s_{\mu_{1,i}, B_i}s_{\mu_{1,i}, B_i}^*)=1$ for $1 \leq i \leq k$. So, for each $1 \leq i \leq k$,   there is $\ld_i \in \Lambda$ such that $\Phi(\mu^{\ld})(s_{\mu_{1,i}, B_i}s_{\mu_{1,i}, B_i}^*)=1$ as $\ld \geq \ld_i$. 
So, for each $1 \leq i \leq k$, $\CP(\mu^{\ld})_{1,i}=\CP(\mu)_{1,i}$ and $B_i \in \eta^{(\ld)}_i$ for all $\ld \geq \ld_i$.
Take $\ld_0:=\max\{\ld_1, \cdots, \ld_k\}$. Then, for each $1 \leq i \leq k$, we have 
 $\mu^{(\ld)}_i=\mu_i$ and $B_i \in \eta^{(\ld)}_i$  as $\ld \geq \ld_0$.
 Thus, $e^{\mu_1^{(\ld)}}_{\eta_1^{(\ld)}}e^{\mu_2^{(\ld)}}_{\eta_2^{(\ld)}} \cdots e^{\mu_k^{(\ld)}}_{\eta_k^{(\ld)}} \in \big(Z^1(\mu_1, B_1) \times Z^1(\mu_2,B_2) \times \cdots \times Z^1(\mu_k, B_k)\big) \cap E^k$ as $\ld \geq \ld_0$.
It means that $\{e^{\mu_1^{(\ld)}}_{\eta_1^{(\ld)}}e^{\mu_2^{(\ld)}}_{\eta_2^{(\ld)}} \cdots e^{\mu_k^{(\ld)}}_{\eta_k^{(\ld)}} \}_{\ld \in \Lambda}$ converses to $e^{\mu_1}_{\eta_1}e^{\mu_2}_{\eta_2} \cdots e^{\mu_k}_{\eta_k}$ in $E^k$.

(3) Suppose  $|\mu| < \infty$ and let $K \subseteq E^1$ be compact. Since $K$ is compact, there exist  basic compact open subsets $Z^1(\bt_1, B_1), Z^1(\bt_2, B_2), \cdots, Z^1(\bt_n, B_n)$ of $E^1$ such that $K \subseteq Z^1(\bt_1, B_1) \cup Z^1(\bt_2, B_2) \cup \cdots \cup Z^1(\bt_n, B_n)$. 
Choose $A \in \eta_{|\mu|}$. Then $\Phi(\mu)(s_{\mu,A}s_{\mu,A}^*)=1$ and $\Phi(\mu)(s_{\mu\bt_i,\theta_{\bt_i}(A) \cap B_i}s_{\mu\bt_i,\theta_{\bt_i}(A) \cap B_i}^*)=0$ for each $1 \leq i \leq n$. Since $\Phi(\mu^{\ld}) \to \Phi(\mu)$, 
one can find $\ld_0 \in \Lambda$ such that  $\Phi(\mu^\ld)(s_{\mu,A}s_{\mu,A}^*)=1$ and $\Phi(\mu^\ld)(s_{\mu\bt_i,\theta_{\bt_i}(A) \cap B_i}s_{\mu\bt_i,\theta_{\bt_i}(A) \cap B_i}^*)=0$  for all $\ld \geq \ld_0$. Thus, for $\ld \geq \ld_0$, we have $|\mu^\ld| \geq |\mu|$, and if $|\mu^\ld| > |\mu|$, then either $\mu^{(\ld)}_{|\mu|+1} \neq \bt_i$ for all $1 \leq i \leq n$, or if $\mu^{(\ld)}_{|\mu|+1} =\bt_i$ for some $1 \leq i \leq n$, then $\theta_{\bt_i}(A) \cap B_i \notin \eta^{(\ld)}_{|\mu|+1}$. 
In the latter case, it follows that $B_i \notin \eta^{(\ld)}_{|\mu|+1}$ since  $\theta_{\bt_i}(A) \in \eta^{(\ld)}_{|\mu|+1}$. 
Thus, if  $|\mu^\ld| > |\mu|$, then $e^{\mu^{(\ld)}_{|\mu|+1}}_{\eta^{(\ld)}_{|\mu|+1}} \notin K$.
\end{proof}

\section{The abelian core and a generalized Uniqueness Theorem}\label{GUT} 

In this section, we define a commutative subalgebra of $C^*(\CB, \CL,\theta, \CI_\af)$, so-called the abelian core, and  show that the ablelian core is isomorphic to the $C^*$-algebra of the interior  of the isotropy group bundle $\operatorname{Iso}(\mathbb{F} \ltimes_\varphi \partial E)$ in $\mathbb{F} \ltimes_\varphi \partial E$. Using this result, we prove a generalized  uniqueness theorem of $C^*(\CB, \CL,\theta, \CI_\af)$, which states that a representation of $C^*(\CB, \CL,\theta, \CI_\af)$ is injective if and only if it is injection on the abelian core.
 
 An element $a \in C^*(\CB, \CL,\theta, \CI_\af)$ is called {\it normal} if $aa^*=a^*a$. We start with observing properties of the  standard generators of $C^*(\CB, \CL,\theta, \CI_\af)$.

\begin{prop}\label{normal and commute}Let $(\CB, \CL,\theta, \CI_\af)$ be a generalized Boolean dynamical system and let $\af, \bt \in \CL^*$ and $A \in \CI_\af \cap \CI_\bt$. The standard generator $s_{\af,A}s_{\bt,A}^*$ is normal  if one of the following holds:
\begin{enumerate}
\item $\af=\bt$;
\item $\af=\bt\gm$ and $(\gm,A)$ is a cycle with no exits;
\item $\bt=\af\gm$ and $(\gm,A)$ is a cycle with no exits. 
\end{enumerate}
\end{prop}

\begin{proof} Let $x=s_{\af,A}s_{\bt,A}^*$. Then $x^*x=s_{\bt,A}s_{\bt,A}^*$ and $xx^*=s_{\af,A}s_{\af,A}^*$. If $\af=\bt$, then clearly $x$ is normal.  Suppose that $\af=\bt\gm$ and $(\gm,A)$ is a cycle with no exits. Then, 
\begin{align*} 
xx^* &=s_{\bt\gm,A}s_{\bt\gm,A}^* \\
&= s_{\bt\gm,\theta_{\gm}(A)}s_{\bt\gm,\theta_{\gm}(A)}^* \\
&=s_{\bt,A}s_{\gm, \theta_{\gm}(A)}s_{\gm, \theta_{\gm}(A)}^*s_{\bt,A}^*\\
&=s_{\bt,A}p_A s_{\bt,A}^* \\
&=s_{\bt,A}s_{\bt,A}^* \\
&=x^*x.
\end{align*}Thus, $x$ is normal. The other case is analogous. 
\end{proof}

By  $D'$ we mean  the commutant of $D$ in $C^*(\CB, \CL,\theta, \CI_\af)$, that is, 
 $$D'=\{x \in C^*(\CB, \CL,\theta, \CI_\af): x (s_{\af,A}s_{\af,A}^*)=(s_{\af,A}s_{\af,A}^*) x ~\text{for all}~ \af \in \CL ~\text{and}~ A \in \CI_\af \}.$$

\begin{prop}\label{prop:abelian core} Let $(\CB, \CL, \theta, \CI_\af)$ be a generalized Boolean dynamical system. If we denote by 
  $M:=M(\CB, \CL,\theta, \CI_\af)$  the subalgebra of $C^*(\CB, \CL,\theta, \CI_\af)$ generated by standard generators $s_{\af,A}s_{\bt,A}^*$ that satisfy:
\begin{enumerate}
\item $\af=\bt$;
\item $\af=\bt\gm$ and $(\gm,A)$ is a cycle with no exits;
\item $\bt=\af\gm$ and $(\gm,A)$ is a cycle with no exits,
\end{enumerate} then, $M$ is commutative satisfying  $D \subseteq M \subseteq D'$.
\end{prop}

\begin{dfn} \label{dfn:abelian core} The subalgebra $M$ defined above is called {\it the abelian core} of $C^*(\CB, \CL,\theta, \CI_\af)$.
\end{dfn}

\begin{proof}[Proof of Proposition~\ref{prop:abelian core}] Clearly, we have $D \subseteq M$. To show that $M$ is commutative, it is enough to check the  following cases:
\begin{itemize}
\item[(i)] Elements of the form (1) commute; it is clear since $s_{\af,A}s_{\af,A}^* \in D$ for $\af \in \CL^*$ and $A \in \CI_\af$.
\end{itemize}
\begin{itemize}
\item[(ii)]  Elements of the form (1) commute with elements of the form (2); let $x=s_{\mu,B}s_{\mu,B}^*$ and $y=s_{\bt\gm,A}s_{\bt,A}^*$, where  $\mu, \bt, \gm \in \CL^*$, $B \in \CI_\mu$, $A \in \CI_\bt$ and   $(\gm,A)$ is a cycle with no exits. If  $\bt=\mu\bt'$, we have that
\begin{align*} 
xy &=s_{\mu,B}s_{\mu,B}^*s_{\bt,A}s_{\gm,A}s_{\bt,A}^* \\
&=s_{\mu,B}s_{\mu,B}^*s_{\mu\bt',A}s_{\gm,A}s_{\bt,A}^* \\
&=s_{\mu,B}s_{\bt',A \cap \theta_{\bt'}(B)}s_{\gm,A}s_{\bt,A}^* \\
&=s_{\mu\bt', A \cap \theta_{\bt'}(B)}s_{\gm,A}s_{\bt,A}^* \\
&=s_{\mu\bt'\gm, A \cap \theta_{\bt'\gm}(B)}s_{\bt,A}^*\\
&=s_{\mu\bt'\gm, A \cap \theta_{\bt'\gm}(B)}s_{\mu\bt',A \cap \theta_{\bt'\gm}(B)}^*,\\
\end{align*}
and that
\begin{align*} 
yx &=s_{\bt\gm,A}s_{\mu\bt',A}^*s_{\mu,B}s_{\mu,B}^* \\
 &=s_{\bt\gm,A}s_{\bt',A \cap  \theta_{\bt'}(B)}^*s_{\mu,B}^* \\
 &=s_{\bt\gm,A}s_{\mu\bt',A \cap  \theta_{\bt'}(B)}^* \\
 &=s_{\mu\bt'\gm,A\cap  \theta_{\bt'}(B)}s_{\mu\bt',A \cap  \theta_{\bt'}(B)}^*.\\
\end{align*}
Since $(\gm,A)$ is a cycle, we have 
$A \cap \theta_{\bt'}(B)= \theta_{\gm}(A \cap \theta_{\bt'}(B))=A \cap \theta_{\bt'\gm}(B)$. Thus, we have  $xy=yx$ if $\bt=\mu\bt'$.

If $\mu=\bt\mu'$ and $ \mu'=\gm^k\gm_{(1)}$ for some $ k \geq 0$ with $ \gm=\gm_{(1)}\gm_{(2)}$, then we have that 
\begin{align*} 
xy&=s_{\mu,B}s_{\bt\mu',B}^*s_{\bt,A}s_{\gm,A}s_{\bt,A}^* \\
&=s_{\mu,B}s_{\mu',B \cap \theta_{\mu'}(A)}^*s_{\gm,A}s_{\bt,A}^* \\
&=s_{\bt\gm^k\gm_{(1)},B}s_{\bt\gm^{k-1}\gm_{(1)},B \cap \theta_{\gm_{(1)}}(A)}^* \\
&=s_{\bt\gm^k\gm_{(1)},B \cap \theta_{\gm_{(1)}}(A)}\big(p_{B \cap \theta_{\gm_{(1)}}(A)}\big)s_{\bt\gm^{k-1}\gm_{(1)},B \cap \theta_{\gm_{(1)}}(A)}^* \\
&=s_{\bt\gm^{k+1}\gm_{(1)},  
\theta_{\gm_{(2)}\gm_{(1)}}(B) \cap \theta_{\gm_{(1)}}(A)}  s_{\bt\gm^{k}\gm_{(1)} ,\theta_{\gm_{(2)}\gm_{(1)}}(B) \cap \theta_{\gm_{(1)}}(A) }^*, \\
\end{align*}
where the last quality follows from $$ p_{B \cap \theta_{\gm_{(1)}}(A)} =s_{\gm_{(2)}\gm_{(1)}, \theta_{\gm_{(2)}\gm_{(1)}}(B) \cap \theta_{\gm_{(1)}}(A)}s_{\gm_{(2)}\gm_{(1)}, \theta_{\gm_{(2)}\gm_{(1)}}(B) \cap \theta_{\gm_{(1)}}(A)}^*,$$ 
and that
\begin{align*}
yx &=s_{\bt\gm,A}s_{\bt,A}^*s_{\bt\mu',B}s_{\mu,B}^* \\
 &=s_{\bt\gm,A}s_{\mu',B \cap \theta_{\mu'}(A)}s_{\mu,B}^* \\
 &=s_{\bt\gm\mu', B \cap \theta_{\mu'}(A)}s_{\mu,B}^* \\
 &=s_{\bt\gm^{k+1}\gm_{(1)}, B \cap \theta_{\gm_{(1)}}(A)}s_{\bt\gm^k\gm_{(1)},B \cap \theta_{\gm_{(1)}}(A) }^*. 
\end{align*}
Now, since $(\gm,A)$ is a cycle with no exit, we have $B \cap \theta_{\gm_{(1)}}(A)=\theta_{\gm_{(2)}\gm_{(1)}}(B \cap \theta_{\gm_{(1)}}(A))$ by \cite[Lemma 2.2]{CaK3}. So, 
\begin{align*}
B \cap \theta_{\gm_{(1)}}(A) 
&=\theta_{\gm_{(2)}\gm_{(1)}}(B)\cap \theta_{\gm_{(1)}\gm_{(2)}\gm_{(1)}}(A) \\
&=\theta_{\gm_{(2)}\gm_{(1)}}(B)\cap \theta_{\gm_{(1)}}(\theta_{\gm_{(1)}\gm_{(2)}}(A)) \\
&=\theta_{\gm_{(2)}\gm_{(1)}}(B)\cap \theta_{\gm_{(1)}}(A).
\end{align*}
Thus, $xy=yx$  if $\mu=\bt\mu'$ and $ \mu'=\gm^k\gm_{(1)}$ for some $k \geq 1$. 
When the case where $\mu=\bt\mu'$ and $\mu'$ is not a beginning of $\gm^\infty$, then $xy=0=yx$. 

Lastly, if $\mu$ and $\bt$ are not comparable, then $xy=0=yx$.
\end{itemize}
\begin{itemize}
\item[(iii)]  Elements of the form (1) commute with elements of the form (3); by (ii), we have $(s_{\mu,B}s_{\mu,B}^*)(s_{\bt\gm,A}s_{\bt,A}^*) = (s_{\bt\gm,A}s_{\bt,A}^*)(s_{\mu,B}s_{\mu,B}^*)$ for $\mu, \bt, \gm \in \CL^*$, $B \in \CI_\mu$, $A \in \CI_\bt$ and   $(\gm,A)$ is a cycle with no exits. Taking adjoint both sides, we have that
$$(s_{\bt,A}s_{\bt\gm,A}^*)(s_{\mu,B}s_{\mu,B}^*)=(s_{\mu,B}s_{\mu,B}^*)(s_{\bt,A}s_{\bt\gm,A}^*)$$ for $\mu, \bt, \gm \in \CL^*$, $B \in \CI_\mu$, $A \in \CI_\bt$ and   $(\gm,A)$ is a cycle with no exits.
\end{itemize}

\begin{itemize}
\item[(iv)]  Elements of the form (2) commute with elements of the form (3);  
let $x=s_{\af\gm,A}s_{\af, A}^*$ and $y=s_{\mu,B}s_{\mu\dt, B}^*$, where $\af,\gm, \mu, \dt \in \CL^*$, $A \in \CI_\af$, $B \in \CI_\mu$, and  $(\gm,A)$ and $(\dt,B)$ are cycles with no exits. 
If $\af=\mu\af'$ and $\af'=\dt\af''$,  then we have that
\begin{align*}
xy &= s_{\af\gm,A}s_{\af, A}^*s_{\mu,B}s_{\mu\dt, B}^* \\
&=s_{\af\gm,A}s_{\mu\af', A}^*s_{\mu,B}s_{\mu\dt, B}^* \\
&=s_{\af\gm,A}s_{\af', A \cap \theta_{\af'}(B)}^*s_{\mu\dt, B}^* \\
&=s_{\mu\dt\af''\gm,A}s_{\mu\dt\dt\af'', A \cap \theta_{\af''}(B)}^*,
\end{align*}
and that 
\begin{align*}
yx &=s_{\mu,B}s_{\mu\dt, B}^*s_{\af\gm,A}s_{\af, A}^*\\
&=s_{\mu,B}s_{\dt, B}^*s_{\mu, B}^*s_{\mu\af',A}s_{\gm,A}s_{\af, A}^*\\
&=s_{\mu,B}s_{\dt, B}^*s_{\af',A \cap \theta_{\af'}(B)}s_{\gm,A}s_{\af, A}^*\\
&=s_{\mu,B}s_{\dt, B}^*s_{\dt\af'',A \cap \theta_{\af'}(B)}s_{\gm,A}s_{\af, A}^*\\
&=s_{\mu,B}s_{\af'',A \cap \theta_{\af'}(B) \cap \theta_{\af''}(B)}s_{\gm,A}s_{\af, A}^*\\
&=s_{\mu\af''\gm , A \cap  \theta_{\af''}(B)}s_{\mu\dt\af'', A}^*.
\end{align*}
If $yx \neq 0$, then  $A \cap \theta_{\af''}(B) \neq \emptyset$, so, $\theta_{\af''}(B) \neq \emptyset$. Thus, $\af''=\dt^k\dt_{(1)}$ for some $k \geq 0$, where $\dt=\dt_{(1)}\dt_{(2)}$.
Also,  since $\theta_{\gm^m}(A \cap \theta_{\af''}(B))= A \cap \theta_{\af''\gm^m}(B) \neq \emptyset$ for each $m \geq 1$, we see that $\af''\gm^m$
is a beginning of $\dt^\infty$ for each $m \geq 1$.
Then we may assume that $\gm=(\dt_{(2)}\dt_{(1)})^n$ for some $n \geq 1$.
Hence, it follows that 
\begin{align*}
xy &=s_{\mu\dt\af''\gm,A}s_{\mu\dt\dt\af'', A \cap \theta_{\af''}(B)}^*\\
&=s_{\mu\dt\dt^k\dt_{(1)}(\dt_{(2)}\dt_{(1)})^n,A}s_{\mu\dt^2\dt^k\dt_{(1)}, A \cap \theta_{\dt_{(1)}}(B)}^*\\
&=s_{\mu\dt\dt^k\dt^{n}\dt_{(1)},A}s_{\mu\dt^2\dt^k\dt_{(1)}, A \cap \theta_{\dt_{(1)}}(B)}^*\\
&=s_{\mu\dt^{k+n+1}\dt_{(1)},A  \cap \theta_{\dt_{(1)}}(B)}s_{\mu\dt^{k+2}\dt_{(1)}, A \cap \theta_{\dt_{(1)}}(B)}^*,
\end{align*}
and that
\begin{align*}
yx &=s_{\mu\af''\gm , A \cap  \theta_{\af''}(B)}s_{\mu\dt\af'', A}^*\\
&=s_{\mu\dt^k\dt_{(1)}(\dt_{(2)}\dt_{(1)})^m , A \cap  \theta_{\dt_{(1)}}(B)}s_{\mu\dt\dt^k\dt_{(1)}, A}^* \\
&=s_{\mu\dt^k\dt^n\dt_{(1)}, A \cap  \theta_{\dt_{(1)}}(B)}s_{\mu\dt\dt^k\dt_{(1)}, A}^* \\
&=s_{\mu\dt^{k+n}\dt_{(1)}, A \cap  \theta_{\dt_{(1)}}(B)}s_{\mu\dt^{k+1}\dt_{(1)}, A}^* \\
&=s_{\mu\dt^{k+n}\dt_{(1)}, A \cap  \theta_{\dt_{(1)}}(B)}\big(p_{A \cap  \theta_{\dt_{(1)}}(B)}\big)s_{\mu\dt^{k+1}\dt_{(1)}, A}^* \\
&=s_{\mu\dt^{k+n+1}\dt_{(1)}, \theta_{\dt_{(2)}\dt_{(1)}}(A) \cap \theta_{\dt_{(1)}}(B) }  s_{\mu\dt^{k+2}\dt_{(1)}  ,   \theta_{\dt_{(2)}\dt_{(1)}}(A) \cap \theta_{\dt_{(1)}}(B) }^*,  
\end{align*}
where the last equality follows from 
$$p_{A \cap  \theta_{\dt_{(1)}}(B)}=s_{\dt_{(2)}\dt_{(1)}, \theta_{\dt_{(2)}\dt_{(1)}}(A) \cap \theta_{\dt_{(1)}}(B) }s_{\dt_{(2)}\dt_{(1)}, \theta_{\dt_{(2)}\dt_{(1)}}(A) \cap \theta_{\dt_{(1)}}(B) }^*.$$
Now, since $((\dt_{(2)}\dt_{(1)})^n,A)$ and $(\dt_{(1)}\dt_{(2)}, B)$ are cycles, we have 
\begin{align*} 
A \cap \theta_{\dt_{(1)}}(B)
&= \theta_{(\dt_{(2)}\dt_{(1)})^n}(A)  \cap \theta_{\dt_{(1)}}(B)\\
&=\theta_{(\dt_{(2)}\dt_{(1)})^{n-1}\dt_{(2)}\dt_{(1)}}(A)  \cap \theta_{\dt_{(1)}}(B)\\
&=\theta_{\dt_{(1)}}\big(\theta_{(\dt_{(2)}\dt_{(1)})^{n-1}\dt_{(2)}}(A)  \cap B \big)\\
&=\theta_{\dt_{(1)}}\big(\theta_{\dt_{(2)}(\dt_{(1)}\dt_{(2)})^{n-1}}(A)  \cap  \theta_{(\dt_{(1)}\dt_{(2)})^{n-1}}(B) \big)\\
&=\theta_{\dt_{(1)}}\big(  \theta_{(\dt_{(1)}\dt_{(2)})^{n-1}}(\theta_{\dt_{(2)}}(A) \cap B) \big)\\
&=\theta_{\dt_{(1)}}(  \theta_{\dt_{(2)}}(A) \cap B ) \\
&=\theta_{\dt_{(2)}\dt_{(1)}}(A) \cap B.
\end{align*}
Therefore, we have $xy=yx$ if $\af=\mu\af'$, $\af'=\dt\af''$ and $ \af''=\dt^k\dt_{(1)}$ for some $k \geq 0$. When $ \af''$ is not a beginning of $\dt^\infty$, then $xy=0=yx$.

If $\af=\mu\af'$, $\dt=\af'\dt'$ and $\dt'=\gm\dt''$ then 
\begin{align*}
yx &=s_{\mu,B}s_{\mu\dt, B}^*s_{\af\gm,A}s_{\af, A}^*\\
&=s_{\mu,B}s_{\mu\af'\gm\dt'', B}^*s_{\mu\af'\gm,A}s_{\af, A}^*\\
&=s_{\mu,B}s_{\dt'', B \cap \theta_{\dt''}(A)}^*s_{\af, A}^*\\
&=s_{\mu,B}s_{\af\dt'', B \cap \theta_{\dt''}(A)}^*.
\end{align*}
Here, if $yx \neq 0$, then $B \cap \theta_{\dt''}(A) \neq \emptyset $, so $\theta_{\dt''}(A) \neq \emptyset $. Thus, $\dt''=\gm^k\gm_{(1)}$ for some $k \geq 0$, where $\gm=\gm_{(1)}\gm_{(2)}$. Also, if $B \cap \theta_{\dt''}(A) \neq \emptyset $, then $\theta_{\dt^n}(B \cap \theta_{\dt''}(A) ) = B \cap  \theta_{\dt''\dt^n}(A) \neq \emptyset$ for all $n \geq 1$. So, $\dt''\dt^n$ is a beginning of $\gm^\infty$ for  each $n \geq 1$. From these facts, one can conclude that $\af'=\gm_{(2)}$, and hence, $\dt=\af'\dt'=\af'\gm\dt''=\gm_{(2)}\gm\gm^k\gm_{(1)}=\gm_{(2)}\gm^{k+1}\gm_{(1)}$. Thus, we have that
\begin{align*}
yx &=s_{\mu,B}s_{\af\dt'', B \cap \theta_{\dt''}(A)}^*\\
&=s_{\mu,B \cap \theta_{\gm_{(1)}}(A)}s_{\af \gm^k\gm_{(1)}, B \cap \theta_{\gm_{(1)}}(A)}^*\\
&=s_{\mu,B \cap \theta_{\gm_{(1)}}(A)}\big(p_{B \cap \theta_{\gm_{(1)}}(A)}\big)s_{\af \gm^k\gm_{(1)}, B \cap \theta_{\gm_{(1)}}(A)}^*\\
&=s_{\mu\gm_{(2)}\gm  , \theta_{\gm_{(2)}\gm}(B) \cap A}s_{\af \gm^{k+2}, \theta_{\gm_{(2)}\gm}(B) \cap A}^*
\end{align*}
where the last equality follows from 
$$p_{B \cap \theta_{\gm_{(1)}}(A)} =s_{\gm_{(2)}\gm, \theta_{\gm_{(2)}\gm}(B \cap \theta_{\gm_{(1)}}(A))}s_{\gm_{(2)}\gm, \theta_{\gm_{(2)}\gm}(B \cap \theta_{\gm_{(1)}}(A))}^*,$$
and that 
\begin{align*}
xy &=s_{\af\gm,A}s_{\mu\dt\af', A \cap \theta_{\af'}(B)}^*\\
&=s_{\mu\gm_{(2)}\gm,A}s_{\mu\gm_{(2)}\gm^{k+1}\gm_{(1)}\gm_{(2)}, A \cap \theta_{\gm_{(2)}}(B)}^* \\
&=s_{\mu\gm_{(2)}\gm, A \cap \theta_{\gm_{(2)}}(B)}s_{\mu\gm_{(2)}\gm^{k+2}, A \cap \theta_{\gm_{(2)}}(B)}^*.
\end{align*}
Now, since $(\gm,A)$ is a cycle, we have
\begin{align*} 
 A \cap \theta_{\gm_{(2)}\gm}(B) 
&=\theta_{\gm}(A) \cap \theta_{\gm_{(2)}\gm}(B) \\
&= \theta_{\gm}(A \cap \theta_{\gm_{(2)}}(B) \\
&= A \cap \theta_{\gm_{(2)}}(B).
\end{align*}
Thus, $xy=yx$ if  $\af=\mu\af'$, $\dt=\af'\dt'$,  $\dt'=\gm\dt''$ and $\dt''=\gm^k\gm_{(1)}$ for some $k \geq 0$. When $\dt''$ is not a beginning of $\gm^\infty$, then $xy=0=yx$.

If $\af=\mu\af'$, $\dt=\af'\dt'$ and $\gm=\dt'\gm'$, we have that
\begin{align*}
yx &=s_{\mu,B}s_{\mu\dt, B}^*s_{\af\gm,A}s_{\af, A}^*\\
&=s_{\mu,B}s_{\mu\af'\dt', B}^*s_{\mu\af'\dt'\gm',A}s_{\af, A}^*\\
&=s_{\mu,B}s_{\gm', A \cap \theta_{\gm'}(B)}s_{\af, A}^*\\
&=s_{\mu\gm', A \cap \theta_{\gm'}(B)}s_{\af, A}^*.
\end{align*}
If $yx \neq 0$, then  $A \cap \theta_{\gm'}(B) \neq \emptyset$, and hence, it must be that $\gm'=\dt^k\dt_{(1)}$ for some $k \geq 0$, where $\dt=\dt_{(1)}\dt_{(2)}$. So, $\gm=\dt'\dt^k\dt_{(1)}$. Also, since $\theta_{\gm^n}(A \cap \theta_{\gm'}(B)) =A \cap \theta_{\gm'\gm^n}(B) \neq \emptyset$ for each $n \geq 1$, we see that 
$\gm'\gm^n$ is a beginning of $\dt^\infty$ for each $n \geq 1$.
From this, we  may assume that $\dt'=\dt_{(2)}$, and hence, $\af'=\dt_{(1)}$ and $\gm=\dt_{(2)}\dt^k\dt_{(1)}$.
Thus, we have that
\begin{align*}
yx &=s_{\mu\gm', A \cap \theta_{\gm'}(B)}s_{\af, A}^* \\
&=s_{\mu\dt^k\dt_{(1)}, A \cap \theta_{\dt_{(1)}}(B)}s_{\mu\dt_{(1)}, A \cap \theta_{\dt_{(1)}}(B)}^* \\
&=s_{\mu\dt^k\dt_{(1)}, A \cap \theta_{\dt_{(1)}}(B)}\big(p_{A \cap \theta_{\dt_{(1)}}(B)}\big)s_{\mu\dt_{(1)}, A \cap \theta_{\dt_{(1)}}(B)}^* \\
&=s_{\mu\dt^{k+1}\dt_{(1)}, \theta_{\dt_{(2)}\dt_{(1)}}(A) \cap \theta_{\dt_{(1)}}(B)} s_{\mu\dt\dt_{(1)} ,\theta_{\dt_{(2)}\dt_{(1)}}(A) \cap \theta_{\dt_{(1)}}(B)}^*,
\end{align*}
where the last equality follows from $$p_{A \cap \theta_{\dt_{(1)}}(B)}=s_{\dt_{(2)}\dt_{(1)}, \theta_{\dt_{(2)}\dt_{(1)}}(A) \cap \theta_{\dt_{(1)}}(B)}s_{\dt_{(2)}\dt_{(1)}, \theta_{\dt_{(2)}\dt_{(1)}}(A) \cap \theta_{\dt_{(1)}}(B)}^*.$$
On the other hand,   we have that
\begin{align*}
xy &=s_{\af\gm,A}s_{\mu\dt\af', A \cap \theta_{\af'}(B)}^* \\
&=s_{\mu\dt_{(1)} \dt_{(2)}\dt^k\dt_{(1)},A \cap \theta_{\dt_{(1)}}(B)}s_{\mu\dt\dt_{(1)}, A \cap \theta_{\dt_{(1)}}(B)}^* \\
&=s_{\mu\dt^{k+1}\dt_{(1)},A \cap \theta_{\dt_{(1)}}(B)}s_{\mu\dt\dt_{(1)}, A \cap \theta_{\dt_{(1)}}(B)}^* 
\end{align*}
Now, since $(\gm, A)$ and $(\dt,B)$ are cycles, it follows that
\begin{align*} A \cap \theta_{\dt_{(1)}}(B) 
& = \theta_{\dt_{(2)}\dt^k\dt_{(1)}}(A) \cap \theta_{\dt_{(1)}}(B) \\
&= \theta_{\dt_{(1)}}( \theta_{\dt_{(2)}\dt^k}(A) \cap B) \\
&=\theta_{\dt_{(1)}}( \theta_{\dt_{(2)}\dt^k}(A) \cap \theta_{\dt^k}(B))\\
&=\theta_{\dt_{(1)}}( \theta_{\dt^k}( \theta_{\dt_{(2)}}(A) \cap B)) \\
&=\theta_{\dt_{(1)}}( \theta_{\dt_{(2)}}(A) \cap B) \\
&= \theta_{\dt_{(2)}\dt_{(1)}}(A) \cap \theta_{\dt_{(1)}}(B).
\end{align*}
Therefore, $xy=yx$ if $\af=\mu\af'$, $\dt=\af'\dt'$, $\gm=\dt'\gm'$ and $\gm'=\dt^k\dt_{(1)}$ for some $k \geq 0$. When $\gm'$ is not a beginning of $\dt^\infty$, then $xy=0=yx$.

By the same arguments, one can show that $xy=yx$ if $\mu=\af\mu'$.
If $\mu$ and $\af$ are not  comparable, then $xy=0=yx$.  
\end{itemize}

\begin{itemize}
\item[(v)]  Elements of the form (3) commute with elements of the form (3);  put $x=s_{\af,A}s_{\af\gm, A}^*$ and $y=s_{\mu,B}s_{\mu\dt, B}^*$,  where $\af,\gm, \mu, \dt \in \CL^*$, $A \in \CI_\af$, $B \in \CI_\mu$, and  $(\gm,A)$ and $(\dt,B)$ are cycles with no exits.   If $\mu=\af\mu'$ and $\gm=\mu'\gm'$, then we have that
\begin{align*}
xy&= s_{\af,A}s_{\af\gm, A}^*s_{\mu,B}s_{\mu\dt, B}^* \\
&=s_{\af,A}s_{\af\mu'\gm', A}^*s_{\af\mu',B}s_{\mu\dt, B}^* \\
&=s_{\af,A}s_{\gm', A \cap \theta_{\gm'}(B)}^*s_{\mu\dt, B}^* \\
&=s_{\af,A}s_{\mu\dt\gm', A \cap \theta_{\gm'}(B)}^*, 
\end{align*}
and that
\begin{align*}
yx&=s_{\mu,B}s_{\mu\dt, B}^*s_{\af,A}s_{\af\gm, A}^*\\
&=s_{\mu,B}s_{\af\mu'\dt, B}^*s_{\af,A}s_{\af\gm, A}^*\\
&=s_{\mu,B}s_{\mu'\dt, B \cap \theta_{\mu'\dt}(A)}^*s_{\af\gm, A}^*\\
&=s_{\mu,B}s_{\af\gm\mu'\dt, B \cap \theta_{\mu'\dt}(A)}^*\\
&=s_{\mu,B}s_{\af\gm\mu'\dt, B \cap \theta_{\mu'}(A)}^*,
\end{align*}
where the last equality follows by $$B \cap \theta_{\mu'\dt}(A)=\theta_\dt(B) \cap \theta_{\mu'\dt}(A)=\theta_\dt(B \cap \theta_{\mu'}(A))=B \cap \theta_{\mu'}(A).$$
If $xy \neq 0 $ and $yx \neq 0$, then $A \cap \theta_{\gm'}(B) \neq \emptyset$ and $B \cap \theta_{\mu'}(A) \neq \emptyset$. So, $ \theta_{\gm'}(B) \neq \emptyset$ and $ \theta_{\mu'}(A) \neq \emptyset$. Thus, 
\begin{align*} \gm'=\dt^k\dt_{(1)}  ~\text{and}~\mu'=\gm^l\gm_{(1)} ~\text{for some}~k,l \geq 0,
\end{align*}
 where $\dt=\dt_{(1)}\dt_{(2)}$ and $\gm=\gm_{(1)}\gm_{(2)}$. 
  We then have 
\begin{align}\label{888}
\gm=\mu'\gm'=\gm^l\gm_{(1)}\gm'.
\end{align}
Also, if $B \cap \theta_{\mu'}(A) \neq \emptyset$, then $\theta_{\dt^m}(B \cap \theta_{\mu'}(A) )=B \cap \theta_{\mu'\dt^m}(A) \neq \emptyset$ for all $m \geq 1$. Thus, 
$\mu'\dt^m$ is a beginning of $\gm^\infty$ for all $m \geq 1$. 

In (\ref{888}), if $l=1$, then $\gm_{(1)}=\gm'=\emptyset$. So,   $\mu'=\gm$ and $\mu=\af\gm$. Thus, we have that
\begin{align*}
xy&=s_{\af,A}s_{\mu\dt\gm', A \cap \theta_{\gm'}(B)}^* \\
&= s_{\af,A \cap B}s_{\af\gm\dt, A \cap B}^* 
\end{align*}
and 
\begin{align*}
yx &=s_{\af\gm,B}s_{\af\gm\dt, B \cap A}^*\\
&=s_{\af\gm, A \cap B}s_{\af\gm\dt, A \cap B}^*.
\end{align*}
So, $xy=yx$. 
In (\ref{888}),  if $l=0$, then $\mu'=\gm_{(1)}$ and $\gm'=\gm_{(2)}$. We then may assume that $\dt=(\gm_{(2)}\gm_{(1)})^n$ for some $n \geq 1$. 
It then follows that
\begin{align*}
&xy \\
&=s_{\af,A}s_{\mu\dt\gm', A \cap \theta_{\gm'}(B)}^* \\
&=s_{\af,A \cap \theta_{\gm_{(2)}}(B)}s_{\af\gm_{(1)}(\gm_{(2)}\gm_{(1)})^n  \gm_{(2)}, A \cap \theta_{\gm_{(2)}}(B)}^* \\
&=s_{\af,A \cap \theta_{\gm_{(2)}}(B)}\big(p_{A \cap \theta_{\gm_{(2)}}(B)}\big)s_{\af\gm^{n+1}, A \cap \theta_{\gm_{(2)}}(B)}^* \\
&=s_{\af,A \cap \theta_{\gm_{(2)}}(B)}\big(s_{\gm_{(1)}, \theta_{\gm(1)}(A \cap \theta_{\gm_{(2)}}(B))}s_{\gm_{(1)}, \theta_{\gm(1)}(A \cap \theta_{\gm_{(2)}}(B))}^*\big)s_{\af\gm^{n+1}, A \cap \theta_{\gm_{(2)}}(B)}^* \\
&=s_{\af\gm_{(1)},     \theta_{\gm(1)}(A ) \cap \theta_{\gm_{(2)}\gm_{(1)}}(B)}s_{\af\gm^{n+1}\gm_{(1)}, \theta_{\gm(1)}(A ) \cap \theta_{\gm_{(2)}\gm_{(1)}}(B)}^*,
\end{align*}
and that
\begin{align*}
yx &=s_{\mu,B}s_{\af\gm\mu'\dt, B \cap \theta_{\mu'}(A)}^*\\
 &=s_{\af\gm_{(1)}, B \cap \theta_{\gm_{(1)}}(A)}s_{\af\gm  \gm_{(1)} (\gm_{(2)}\gm_{(1)})^n, B \cap \theta_{\gm_{(1)}}(A)}^*\\
 &=s_{\af\gm_{(1)}, B \cap \theta_{\gm_{(1)}}(A)}s_{\af\gm^{n+1}\gm_{(1)}, B \cap \theta_{\gm_{(1)}}(A)}^*.
\end{align*}
Now, since $(\gm,A)$  and $((\gm_{(2)}\gm_{(1)})^n,B)$ are cycles, we have 
\begin{align*} 
B \cap \theta_{\gm_{(1)}}(A)
&=\theta_{(\gm_{(2)}\gm_{(1)})^n}(B) \cap \theta_{\gm_{(1)}}(A)\\
&=\theta_{(\gm_{(2)}\gm_{(1)})^{n-1}\gm_{(2)}\gm_{(1)}}(B) \cap \theta_{\gm_{(1)}}(A)\\
&=\theta_{\gm_{(1)}}(\theta_{(\gm_{(2)}\gm_{(1)})^{n-1}\gm_{(2)}}(B) \cap A)\\
&=\theta_{\gm_{(1)}}(\theta_{\gm_{(2)}(\gm_{(1)}\gm_{(2)})^{n-1}}(B) \cap \theta_{(\gm_{(1)}\gm_{(2)})^{n-1}}(A))\\
&=\theta_{\gm_{(1)}}(\theta_{(\gm_{(1)}\gm_{(2)})^{n-1}}(\theta_{\gm_{(2)}}(B) \cap A))\\
&=\theta_{\gm_{(1)}}(\theta_{\gm_{(2)}}(B) \cap A)\\
&=\theta_{\gm_{(2)}\gm_{(1)}}(B) \cap \theta_{\gm_{(1)}}(A).
\end{align*}
 Therefore,  $xy=yx$ if  $\mu=\af\mu'$ and $\gm=\mu'\gm'$.

If $\mu=\af\mu'$ and $\mu'=\gm\mu''$, then we have that
\begin{align*} 
xy &= s_{\af,A}s_{\af\gm,A}^*s_{\af\gm\mu'',B}s_{\mu\dt,B}^* \\
&= s_{\af,A}s_{\mu'',B \cap \theta_{\mu''}(A)}s_{\mu\dt,B}^* \\
&= s_{\af\mu'',B \cap \theta_{\mu''}(A)}s_{\mu\dt,B}^*, 
\end{align*}
and that 
\begin{align*}
yx &=s_{\mu,B}s_{\af\mu'\dt,B}^*s_{\af,A}s_{\af\gm,A}^* \\
&=s_{\mu,B}s_{\mu'\dt,B \cap \theta_{\mu'\dt}(A)}^*s_{\af\gm,A}^* \\
&=s_{\mu,B}s_{\af\gm\mu'\dt,B \cap \theta_{\mu'\dt}(A)}^* \\
&=s_{\mu,B}s_{\af\gm\mu'\dt,B \cap \theta_{\mu'}(A)}^*\\
&=s_{\mu,B}s_{\af\gm\mu'\dt,B \cap \theta_{\gm\mu''}(A)}^*\\
&=s_{\mu,B}s_{\af\gm\mu'\dt,B \cap \theta_{\mu''}(A)}^*
\end{align*}
since $(\dt,B)$ and $(\gm,A)$ are cycles. 
If $yx \neq 0$, then $B \cap \theta_{\mu''}(A) \neq \emptyset$, so $\theta_{\mu''}(A) \neq \emptyset$. Thus, $\mu''=\gm^k\gm_{(1)}$ for some $k \geq 0$, where $\gm=\gm_{(1)}\gm_{(2)}$.
Also, if $B \cap \theta_{\mu''}(A) \neq \emptyset$, then $\theta_{\dt^m}(B \cap \theta_{\mu''}(A) )=B \cap \theta_{\mu''\dt^m}(A) \neq \emptyset $ for each $m \geq 1$, and hence, $\mu''\dt^m$ is a beginning of $\gm^\infty$ for each $n \geq 1$.
We then may assume that $\dt=(\gm_{(2)}\gm_{(1)})^n$ for some $n \geq 1$. Hence, it follows that 
\begin{align*}
yx &=s_{\af\gm^{k+1}\gm_{(1)},B}s_{\af\gm^2\gm^k\gm_{(1)}(\gm_{(2)}\gm_{(1)})^n,B \cap \theta_{\mu''}(A)}^*  \\
&=s_{\af\gm^{k+1}\gm_{(1)},B \cap \theta_{\gm_{(1)}}(A)}s_{\af\gm^{k+n+2}\gm_{(1)},B \cap \theta_{\gm_{(1)}}(A)}^*, 
\end{align*}
and that
\begin{align*}
xy &= s_{\af\gm^k\gm_{(1)},B \cap \theta_{\gm_{(1)}}(A)}s_{\af\gm\gm^k\gm_{(1)} (\gm_{(2)}\gm_{(1)})^n,B}^* \\
&= s_{\af\gm^k\gm_{(1)},B \cap \theta_{\gm_{(1)}}(A)}s_{\af\gm^{k+1}\gm^n \gm_{(1)},B \cap \theta_{\gm_{(1)}}(A)}^* \\
&= s_{\af\gm^k\gm_{(1)},B \cap \theta_{\gm_{(1)}}(A)}\big(p_{B \cap \theta_{\gm_{(1)}}(A)}\big)s_{\af\gm^{k+1+n}\gm_{(1)},B \cap \theta_{\gm_{(1)}}(A)}^* \\
&= s_{\af\gm^{k+1}\gm_{(1)},\theta_{\gm_{(2)}\gm_{(1)}}(B) \cap \theta_{\gm_{(1)}}(A)} s_{\af\gm^{k+n+2} \gm_{(1)},\theta_{\gm_{(2)}\gm_{(1)}}(B) \cap \theta_{\gm_{(1)}}(A)}^*,
\end{align*}
where the last equality follows from $$p_{B \cap \theta_{\gm_{(1)}}(A)}= s_{\gm_{(2)}\gm_{(1)}, \theta_{\gm_{(2)}\gm_{(1)}}(B \cap \theta_{\gm_{(1)}}(A))}s_{\gm_{(2)}\gm_{(1)}, \theta_{\gm_{(2)}\gm_{(1)}}(B \cap \theta_{\gm_{(1)}}(A))}^*.$$
Now, since $(\gm,A)$ and $((\gm_{(2)}\gm_{(1)})^n,B)$ are cycles, one again can have that $$B \cap \theta_{\gm_{(1)}}(A)=\theta_{\gm_{(2)}\gm_{(1)}}(B) \cap \theta_{\gm_{(1)}}(A).$$
Therefore, $xy=yx$ if $\mu=\af\mu'$ and $\mu'=\gm\mu''$.

The similar arguments give that $xy=yx$ if $\af=\mu\af'$. If $\mu$ and $\af$ are not  comparable, then $xy=0=yx$.  

\end{itemize}

\begin{itemize}
\item[(vi)]  Elements of the form (2) commute with elements of the form (2); by (v), we have $(s_{\af,A}s_{\af\gm, A}^*)(s_{\mu,B}s_{\mu\dt, B}^*)=(s_{\mu,B}s_{\mu\dt, B}^*)(s_{\af,A}s_{\af\gm, A}^*)$  for $\af,\gm, \mu, \dt \in \CL^*$, $A \in \CI_\af$, $B \in \CI_\mu$, and  $(\gm,A)$ and $(\dt,B)$ are cycles with no exits. Taking adjoint both sides, we have that
$$(s_{\mu\dt, B}s_{\mu,B}^*)(s_{\af\gm, A}s_{\af,A}^*)=(s_{\af\gm, A}s_{\af,A}^*)(s_{\mu\dt, B}s_{\mu,B}^*)$$
for $\af,\gm, \mu, \dt \in \CL^*$, $A \in \CI_\af$, $B \in \CI_\mu$, and  $(\gm,A)$ and $(\dt,B)$ are cycles with no exits.
\end{itemize}
The calculations (ii) and (iii) also shows that $M \subseteq D'$. So, we are done.
\end{proof}

We now identify the abelian core $M$ with the $C^*$-algebra of the interior of the isotropy group bundle $\operatorname{Iso} (\mathbb{F} \ltimes_\varphi \partial E)$ of $\mathbb{F} \ltimes_\varphi \partial E$.
We first characterize the isotropy group bundle $\operatorname{Iso} (\mathbb{F} \ltimes_\varphi \partial E)=\{(t, \mu) \in\mathbb{F} \ltimes_\varphi \partial E : \varphi_{t^{-1}}(\mu)=\mu \}$.

\begin{lem}\label{char:isotropy} For $(t, \mu) \in \mathbb{F} \ltimes_\varphi \partial E$ such that $t \neq \emptyset$ and $\mu=(e^{\mu_i}_{\xi_i})$,  we have $(t, \mu) \in \operatorname{Iso} (\mathbb{F} \ltimes_\varphi \partial E)$ if and only if there exist $\dt, \gm \in \CL^*$ such that  $t = \dt\gm\dt^{-1}$ or $t=\dt \gm^{-1} \dt^{-1}$,  $\CP(\mu)=\dt \gm^\infty$ and  $e^{\gm_1}_{\xi_{|\dt|+1}} \cdots e^{\gm_{|\gm|}}_{\xi_{|\dt|+|\gm|}}$ is a loop at $d(e^{\dt_{|\dt|}}_{\xi_{|\dt|}})$ so that $\mu=e^{\dt_1}_{\xi_1} \cdots e^{\dt_{|\dt|}}_{\xi_{|\dt|}}(e^{\gm_1}_{\xi_{|\dt|+1}} \cdots e^{\gm_{|\gm|}}_{\xi_{|\dt|+|\gm|}})^{\infty} $.
\end{lem}

\begin{proof} 
($\Rightarrow$) If $(t, \mu) \in \operatorname{Iso} (\mathbb{F} \ltimes_\varphi \partial E)$, then $\mu=\varphi_{t^{-1}}(\mu)$, where $t=\af\bt^{-1}$ for some $\af,\bt \in \CW^*$. So, we have that $\varphi_{\bt^{-1}}(\mu)=\varphi_{\af^{-1}}(\mu)$. If $|\af|< |\bt|$, then there exist $\gm \in \CL^*$ such that  $\bt=\af\gm$, $\CP(\mu)=\af\gm^{\infty}$ and $t=\af\gm^{-1}\af^{-1}$.  If $|\bt|<|\af|$, then there exist $\gm \in \CL^*$ such that $\af=\bt\gm$, $\CP(\mu)=\bt\gm^{\infty}$ and $t=\bt\gm\bt^{-1}$.
 So, one can conclude that there exist $\dt, \gm \in \CL^*$ such that  $t = \dt\gm\dt^{-1}$ or $t=\dt \gm^{-1} \dt^{-1}$ and  $\CP(\mu)=\dt \gm^\infty$. Now,  put $m=|\dt|$ and $k=|\gm|$ and say  
 $\mu=e^{\dt_1}_{\xi_1} \cdots e^{\dt_{m}}_{\xi_{m}}e^{\gm_1}_{\xi_{m+1}} \cdots e^{\gm_{k}}_{\xi_{m+k}} \cdots e^{\gm_1}_{\xi_{m+(n-1)k+1}} \cdots e^{\gm_{k}}_{\xi_{m+nk}}  \cdots $.

If $t=\dt\gm\dt^{-1}$, then $\mu=\theta_{\dt\gm^{-1}\dt^{-1}}(\mu)$, and hence, we have 
\begin{align}\label{base index}\xi_{m+i}=\xi_{m+(n-1)k+i}
\end{align}
for all $n \geq 2$ and  all $i=1, \cdots, k$.

From (\ref{base index}),  for each $n \geq 1$, we have that
$$d(e^{\gm_k}_{\xi_{m+nk}})=r(e^{\gm_1}_{\xi_{m+nk+1}})=r(e^{\gm_1}_{\xi_{m+1}})=d(e^{\dt_k}_{\xi_{m}}).$$
Thus,  $e^{\gm_1}_{\xi_{m+1}} \cdots e^{\gm_{k}}_{\xi_{m+k}}$ is a loop at $d(\xi_m)$ and $\mu=e^{\dt_1}_{\xi_1} \cdots e^{\dt_{m}}_{\xi_{m}}(e^{\gm_1}_{\xi_{m+1}} \cdots e^{\gm_{k}}_{\xi_{m+k}})^{\infty} $.

($\Leftarrow$) Say $\mu=e^{\dt_1}_{\xi_1} \cdots e^{\dt_{m}}_{\xi_{m}}(e^{\gm_1}_{\xi_{m+1}} \cdots e^{\gm_{k}}_{\xi_{m+k}})^{\infty}$, where $m=|\dt|$ and $k=|\gm|$.  First, for $t=\dt\gm^{-1}\dt^{-1}$, we show that $\mu=\varphi_{t^{-1}}(\mu)=\varphi_{\dt\gm\dt^{-1}}(\mu)$. We observe that 
$\varphi_{\dt^{-1}}(\mu)=(e^{\gm_1}_{\xi_{m+1}} \cdots e^{\gm_{k}}_{\xi_{m+k}})^{\infty}$, and that $\varphi_{\gm}(\varphi_{\dt^{-1}}(\mu))=e^{\gm_1}_{\eta_1} \cdots e^{\gm_k}_{\eta_k}(e^{\gm_1}_{\xi_{m+1}} \cdots e^{\gm_{k}}_{\xi_{m+k}})^{\infty}$, where 
\begin{align*} \eta_k &:=g_{(\gm_k)\emptyset}(h_{[\dt_{m}]\emptyset}(\xi_{m})) , \\
\eta_{k-1}&:=g_{(\gm_{k-1})\emptyset}(f_{\emptyset[\gm_k]}(\eta_k)), \\
\eta_{k-2}&:=g_{(\gm_{k-2})\emptyset}(f_{\emptyset[\gm_{k-1}]}(\eta_{k-1})), \\
& \vdots \\
\eta_2 &:= g_{(\gm_2)\emptyset}(f_{\emptyset[\gm_3]}(\eta_3)),\\
\eta_1&:= g_{(\gm_1)\emptyset}(f_{\emptyset[\gm_2]}(\eta_2)).
\end{align*}
 Since $d(e^{\gm_k}_{\eta_k})=d(e^{\gm_{k}}_{\xi_{m+k}})$, we have $\eta_k=\xi_{m+k}$. So, $e^{\gm_k}_{\eta_k}=e^{\gm_{k}}_{\xi_{m+k}}$. Then again, since  $d(e^{\gm_{k-1}}_{\eta_{k-1}})=d(e^{\gm_{k-1}}_{\xi_{m+k-1}})$, we have $e^{\gm_{k-1}}_{\eta_{k-1}}=e^{\gm_{k-1}}_{\xi_{m+k-1}}$.
Continuing this precess, we have $e^{\gm_1}_{\eta_1} \cdots e^{\gm_k}_{\eta_k}=e^{\gm_1}_{\xi_1} \cdots e^{\gm_k}_{\xi_k}$. Thus, $\varphi_{\gm}(\varphi_{\dt^{-1}}(\mu))=(e^{\gm_1}_{\xi_{m+1}} \cdots e^{\gm_{k}}_{\xi_{m+k}})^{\infty}$. 
Then, since $g_{(\dt_m)\emptyset}(h_{[\dt_{m}]\emptyset}(\xi_{m}))=\xi_{m}$, one can see that $\varphi_{\dt\gm\dt^{-1}}(\mu)=\varphi_{\dt}(\varphi_{\gm\dt^{-1}}(\mu))=\mu$. Thus,  $(t,\mu) \in  \operatorname{Iso} (\mathbb{F} \ltimes_\varphi \partial E)$.

For $t=\dt\gm\dt^{-1}$, a similar argument gives $\varphi_{\dt\gm^{-1}\dt^{-1}}(\mu)=\mu$. So, $(t,\mu) \in  \operatorname{Iso} (\mathbb{F} \ltimes_\varphi \partial E)$.
\end{proof}

\begin{prop}\label{char:abelian core} The abelian core $M$  is isomorphic to $C^*(\operatorname{Iso}(\mathbb{F} \ltimes_\varphi \partial E)^{\circ})$.
\end{prop}

\begin{proof} Let $\kp : C^*(\CB,\CL,\theta,\CI_\af)  \rightarrow C^*(\mathbb{F} \ltimes_\varphi \partial E)$ be an isomorphism given by 
$\kp(s_{\af,A}s_{\bt,A}^*)=1_{ \{\af\bt^{-1}\} \times \CN(\af,A)}$ for $\af, \bt \in \CL^*$ and $\emptyset \neq A \in \CI_\af \cap \CI_\bt$. We show that $$\kappa(M)=C^*(\operatorname{Iso}(\mathbb{F} \ltimes_\varphi \partial E)^{\circ}).$$
Since $D$  is isomorphic to $C_0(\partial E) (\cong C_0((\mathbb{F} \ltimes_\varphi \partial E)^{(0)}))$, one can see that $\kp(D) \subseteq C^*(\operatorname{Iso}(\mathbb{F} \ltimes_\varphi \partial E)^{\circ})$. Now, let $\af, \bt \in \CL^*$ and $\emptyset \neq A \in \CI_\af \cap \CI_\bt$ be such that $\af=\bt\gm$ and $(\gm,A)$ is a cycle with no exits. Note  that  $\{\af\bt^{-1}\} \times \CN(\af, A) \subseteq \mathbb{F} \ltimes_\varphi \partial E$.
We claim that $$ \{\af\bt^{-1}\} \times \CN(\af, A) \subseteq \operatorname{Iso}(\mathbb{F} \ltimes_\varphi \partial E)^{\circ}.$$
Choose $(t, \mu) \in \{\af\bt^{-1}\} \times \CN(\af, A)$ and say $\mu=(e^{\mu_i}_{\xi_i})$.  
Then $t=\af\bt^{-1}=\bt\gm\bt^{-1}$
and  $A \in \xi_{|\bt\gm|}$. Since $(\gm,A)$ is a cycle, we have $(\gm, h_{[\gm_{|\gm|}]\emptyset}(\xi_{|\bt\gm|})) $ is an ultrafilter cycle. 
So, $e^{\gm_1}_{\xi_{|\bt|+1}} \cdots  e^{\gm_{|\gm|}}_{\xi_{|\bt|+|\gm|}} $ is a loop by Lemma \ref{char2:ultrafilter cycle}. Also, since the cycle $(\gm,A)$ has no exit, the loop $e^{\gm_1}_{\xi_{|\bt|+1}} \cdots  e^{\gm_{|\gm|}}_{\xi_{|\bt|+|\gm|}} $ has no entrance (see the ``if'' part of the proof of \cite[Proposition 3.5]{CaK3}). Hence, $\mu =e^{\bt_1}_{\xi_1} \cdots e^{\bt_{|\bt|}}_{\xi_{|\bt|}}(e^{\gm_1}_{\xi_{|\bt|+1}} \cdots e^{\gm_{|\gm|}}_{\xi_{|\bt|+|\gm|}})^{\infty} $. So, $(t, \mu) \in \operatorname{Iso}(\mathbb{F} \ltimes_\varphi \partial E)$ by Lemma \ref{char:isotropy}.
Thus, $\{\af\bt^{-1}\} \times \CN(\af, A) \subseteq  \operatorname{Iso}(\mathbb{F} \ltimes_\varphi \partial E)$. 
On the other hand, since $\{\af\bt^{-1}\} \times \CN(\af, A)$ is  open  in $\mathbb{F} \ltimes_\varphi \partial E$, we have  $\{\af\bt^{-1}\} \times \CN(\af, A) \subseteq \operatorname{Iso}(\mathbb{F} \ltimes_\varphi \partial E)^{\circ}$. Hence, 
$ \{\af\bt^{-1}\} \times \CN(\af, A) \subseteq \operatorname{Iso}(\mathbb{F} \ltimes_\varphi \partial E)^{\circ}.$

 Hence, we have $
\kp(s_{\af,A}s_{\bt,A}^*)=1_{ \{\af\bt^{-1}\} \times \CN(\af,A)}  \in C_c(\operatorname{Iso}(\mathbb{F} \ltimes_\varphi \partial E)^{\circ})$.
By the same argument as above, one also can see that $\kp(s_{\bt,A}s_{\af,A}^*)=1_{\{\bt\af^{-1}\} \times \CN(\bt, A)} \in C_c(\operatorname{Iso}(\mathbb{F} \ltimes_\varphi \partial E)^{\circ}).$
Then, the continuity of $\kp$ implies that $\kp(M) \subseteq C^*(\operatorname{Iso}(\mathbb{F} \ltimes_\varphi \partial E)^{\circ})$.

  To prove that $C^*(\operatorname{Iso}(\mathbb{F} \ltimes_\varphi \partial E)^{\circ})  \subseteq \kp(M)$, we only need to show that $$1_{ \{\af\bt^{-1}\}\times \CN(\af,A)} \in \kp(M),$$
  where $\{\af\bt^{-1}\}\times \CN(\af,A)$  is a compact open bisection in $\mathbb{F} \ltimes_\varphi \partial E$ such that $$ \{\af\bt^{-1}\}\times \CN(\af,A) \subseteq \operatorname{Iso}(\mathbb{F} \ltimes_\varphi \partial E)^{\circ}.$$
Let $(t, \mu) \in \{\af\bt^{-1}\}\times \CN(\af,A) \subseteq \operatorname{Iso}(\mathbb{F} \ltimes_\varphi \partial E)^{\circ}$ and say $\mu=(e^{\mu_i}_{\xi_i})$. Then, by Lemma \ref{char:isotropy},
there exist $\dt, \gm \in \CL^*$ such that  $t = \dt\gm\dt^{-1}$ or $t=\dt \gm^{-1} \dt^{-1}$,  $\CP(\mu)=\dt \gm^\infty$ and  $e^{\gm_1}_{\xi_{|\dt|+1}} \cdots e^{\gm_{|\gm|}}_{\xi_{|\dt|+|\gm|}}$ is a loop at $d(e^{\dt_{|\dt|}}_{\xi_{|\dt|}})$ and $\mu=e^{\dt_1}_{\xi_1} \cdots e^{\dt_{|\dt|}}_{\xi_{|\dt|}}(e^{\gm_1}_{\xi_{|\dt|+1}} \cdots e^{\gm_{|\gm|}}_{\xi_{|\dt|+|\gm|}})^{\infty}$.
Put $\eta= d(e^{\dt_{|\dt|}}_{\xi_{|\dt|}})$. Then, since $(\gm, \eta) $ is an ultrafilter cycle by Lemma \ref{char2:ultrafilter cycle} and $A \in \eta$,  we have  $A \cap \theta_{\gm}(A) \neq \emptyset$ by Lemma \ref{char:ultrafilter cycle}. We claim that $(\gm, A \cap \theta_{\gm}(A))$ is a cycle with no exits. 
Let $\emptyset \neq B \subseteq A \cap \theta_{\gm}(A)$.
If  $B \setminus \theta_{\gm}(B) \neq \emptyset$, choose an ultrafilter $\chi \in \widehat{\CB}$ such that $B \setminus \theta_{\gm}(B) \in \chi$. Then, $B, A, \theta_{\gm}(A) \in \chi$.
Since $A \in \CI_\af$, we have $\theta_{\gm}(A) \in \CI_{\af\gm}$. Thus, 
 $\chi \cap \CI_{\af\gm} \neq \emptyset$. Notice here that we have either $\af=\dt$ or $\af=\dt\gm$. Put 
\begin{align*} \chi_{|\af\gm|} &:=g_{((\af\gm)_{|\af\gm|})\emptyset}(\chi) , \\
\chi_{|\af\gm|-1}&:=g_{((\af\gm)_{|\af\gm|-1})\emptyset}(f_{\emptyset[(\af\gm)_{|\af\gm|}]}(\chi_{|\af\gm|})), \\
\chi_{|\af\gm|-2}&:=g_{((\af\gm)_{|\af\gm|-2})\emptyset}(f_{\emptyset[(\af\gm)_{{|\af\gm|}-1}]}(\chi_{{|\af\gm|}-1})), \\
& \vdots \\
\chi_2 &:= g_{((\af\gm)_2)\emptyset}(f_{\emptyset[(\af\gm)_3]}(\chi_3)),\\
\chi_1&:= g_{((\af\gm)_1)\emptyset}(f_{\emptyset[(\af\gm)_2]}(\chi_2)).
\end{align*}
Then we have  a path $\nu \in \partial E$ such that 
$\nu_{[1,|\af\gm|]}=e^{\af_1}_{\chi_1} \cdots e^{\af_{|\af|}}_{\chi_{|\af|}}e^{\gm_1}_{\chi_{|\af|+1}} \cdots e^{\gm_{|\gm|}}_{\chi_{|\af\gm|}}$.
Note that $\nu \in U_{\af\bt^{-1}}$ since $A \in \chi_{|\af|} \cap \CI_{\bt}$. Thus,  we have   $(\af\bt^{-1}, \nu) \in {\{\af\bt^{-1}\}} \times \CN(\af,A) \subseteq \operatorname{Iso}(\mathbb{F} \ltimes_\varphi \partial E)^{\circ}$. 
Hence, by Lemma \ref{char:isotropy}, 
$e^{\gm_1}_{\chi_{|\af|+1}} \cdots e^{\gm_{|\gm|}}_{\chi_{|\af\gm|}}$ is a loop at $\chi$. So,  $(\gm, \chi )$ is an ultrafilter cycle. Then Lemma \ref{char:ultrafilter cycle} implies that $ \big(B \setminus \theta_{\gm}(B) \big)  \cap  \big(\theta_{\gm}(B \setminus \theta_{\gm}(B)) \big) \neq \emptyset$. But, 
 $$ \big(B \setminus \theta_{\gm}(B) \big)  \cap  \big(\theta_{\gm}(B \setminus \theta_{\gm}(B))\big) =\big(B \setminus \theta_{\gm}(B) \big)  \cap  \big(\theta_{\gm}(B) \setminus \theta_{\gm^2}(B) \big) =\emptyset,$$
 a contradiction. 
Therefore, $B \setminus \theta_{\gm}(B)=\emptyset$, and hence, $ B \subseteq \theta_{\gm}(B)$. 

Now, if $ \theta_{\gm}(B) \setminus B \neq \emptyset$, then choose an ultrafilter $\chi \in \widehat{\CB}$ such that $ \theta_{\gm}(B) \setminus B \in \chi$. Then  the same argument as above implies that  $(\gm, \chi )$ is an ultrafilter cycle. Then again Lemma \ref{char:ultrafilter cycle} implies that $ \big( \theta_\gm(B) \setminus B \big) \cap \big( \theta_{\gm}(\theta_\gm(B) \setminus B) \big) \neq \emptyset$. But, 
$$\big( \theta_\gm(B) \setminus B \big) \cap \big( \theta_{\gm}(\theta_\gm(B) \setminus B) \big) = \big( \theta_\gm(B) \setminus B \big) \cap \big( \theta_{\gm^2}(B) \setminus \theta_\gm(B) \big) =\emptyset,$$
a contradiciton. Thus, $\theta_{\gm}(B) \setminus B =\emptyset$, so we have $B =\theta_{\gm}(B)$. Therefore, $(\gm, A \cap \theta_{\gm}(A))$ is a cycle.  

If $(\gm, A \cap \theta_{\gm}(A))$ has an exit, then there exist $0 \leq k \leq |\gm|$ and $\emptyset \neq C \in \CB$  such that $C \subseteq \theta_{\gm_1 \cdots \gm_k}(A \cap \theta_{\gm}(A))$  and $\Delta_C \neq \{\gm_{k+1}\}$.
If $\Delta_C =\emptyset$, choose $\chi \in \widehat{\CB}$ such that $C \in \chi$.
Since $\theta_{\gm\gm_1 \cdots \gm_k}(A) \in \chi$ and $A \in \CI_\af$, 
we have  $\chi \cap \CI_{\af\gm\gm_1 \cdots \gm_k} \neq \emptyset$. Consider a path $\nu$ such that 
$\CP(\nu)_{1, |\af\gm\gm_1 \cdots \gm_k|}=\af\gm\gm_1 \cdots \gm_k$ and $d(\nu_{[1,|\af\gm\gm_1 \cdots \gm_k|]})=\chi$. Claim that $\chi \in E^0_{sg}$. We show that $D \notin \CB_{reg}$ for all $D \in \chi$. If not, there is $D \in \chi$ such that $D \in \CB_{reg}$. Since $C \in \chi$ with $\Delta_C=\emptyset$, we have $D \cap C \in \chi$ and $\Delta_{D \cap C}=\emptyset$. This contradicts to $D \in \CB_{reg}$.  So, $D \notin \CB_{reg}$ for all $D \in \chi$. Hence, $\chi \in E^0_{sg}$ by \cite[Lemma 7.9]{CasK1}.
On the other hand, one can see that $(\af\bt^{-1}, \nu) \in  \{\af\bt^{-1}\}\times \CN(\af,A) \subseteq \operatorname{Iso}(\mathbb{F} \ltimes_\varphi \partial E)^{\circ}$. So, it must be that $\CP(\nu)=\dt\gm^\infty$ for some $\dt \in \CL^*$, this is not the case. 

If $b \in \Delta_C \setminus \{\gm_{k+1}\}$, choose $\chi \in \widehat{\CB}$ such that $\emptyset \neq \theta_b(C) \in \chi$.
One then have a path $\nu$ such that $\CP(\nu)_{1, |\af\gm\gm_1 \cdots \gm_k b|}=\af\gm\gm_1 \cdots \gm_k b$. But, since $(\af\bt^{-1}, \nu) \in  \{\af\bt^{-1}\}\times \CN(\af,A) \subseteq \operatorname{Iso}(\mathbb{F} \ltimes_\varphi \partial E)^{\circ} $, we have $\CP(\nu)=\dt\gm^\infty$ for some $\dt \in \CL^*$, this is not the case. Therefore, $(\gm, A \cap \theta_{\gm}(A))$ has no exits. 

Now, we see that 
\begin{align*}
&\kp(s_{\af, A \cap \theta_{\gm}(A)}s_{\bt,A \cap \theta_{\gm}(A)}^*)\\
&=1_{ \{\af\bt^{-1}\}\times \CN(\af, A \cap \theta_{\gm}(A))}\\
&= 1_{ \{\af\bt^{-1}\}\times \CN(\af,A)}, 
\end{align*}
where $\af=\bt\gm$ or $\bt=\af\gm$ and $(\gm, A \cap \theta_{\gm}(A))$ is a cycle with no exit. 
Hence, $1_{ \{\af\bt^{-1}\}\times \CN(\af,A)} \in \kp(M).$ Therefore, $C^*(\operatorname{Iso}(\mathbb{F} \ltimes_\varphi \partial E)^{\circ})  \subseteq \kp(M)$.
\end{proof}

We now state our main result. It is a generalization of \cite[Theorem 3.12]{NR2012} and  \cite[Theorem 6.11]{CGW2020}.

\begin{thm} \label{GUT}Let $(\CB, \CL,\theta, \CI_\af)$ be a generalized Boolean dynamical system such that $\CB$ and $\CL$ are countable. If $\pi: C^*(\CB,\CL, \theta, \CI_\af) \to A$ is a $*$-homomorphism into a $C^*$-algebra $A$, then $\pi$ is injective if and only if the restriction of $\pi$ to $M$ is injective. 
\end{thm}

\begin{proof} It follows by Proposition \ref{char:abelian core} and \cite[Theorem 3.1(b)]{BNRSW}.
\end{proof}

\begin{remark}If $(\CB, \CL,\theta)$ satisfies Condition (L), then there are no cycle with no exits, so we immediately have the usual Cuntz--Krieger uniqueness theorem given in \cite[Theorem 3.6]{CaK3}.
\end{remark}


We conclude this section by observing that if the underlying Boolean dynamical system $(\CB, \CL, \theta)$ satisfies Condition $(L)$, then   the diagonal subalgebra is  a maximal abelian subalgebra (MASA), and it coincide with  the abelian core. As a result, the abelian core $M$ is a MASA in $C^*(\CB, \CL, \theta, \CI_\af)$ if $(\CB, \CL, \theta)$ satisfies Condition $(L)$.

To begin, we describe Condition (L) in terms of the groupoid $\mathbb{F} \ltimes_\varphi   \partial E$  and the partial action $\Phi=(\{U_t\}_{t \in G}, \{\varphi_t\}_{t \in G})$.
Recall that  the partial action $\Phi$ is called {\it topologically free} (\cite[Definition 6.2]{CW2020}) if the set of fixed point $\operatorname{Fix}(t):=\{\mu \in \partial E: \mu \in U_{t^{-1}} ~\text{and}~ \varphi_t(\mu)=\mu\}$ has empty interior for all $t \in \mathbb{F} \setminus \{ \emptyset \}$.

\begin{prop}\label{equiv:Condition (L)} Let $(\CB, \CL,\theta, \CI_\af)$ be a generalized Boolean dynamical system. Consider the following:
\begin{enumerate}
\item $(\CB, \CL,\theta)$ satisfies Condition (L).
\item  $\Phi$ is topologically free.
\item   $\mathbb{F} \ltimes_\varphi \partial E$ is topologically principal.
\item   $\mathbb{F} \ltimes_\varphi \partial E$ is effective
\end{enumerate}
Then, we have (1)$\iff$(2)$\iff$(3)$\implies$(4). Moreover, if $\CB$ and $\CL$ are countable, we have (4)$\implies$(3).
\end{prop}

\begin{proof}  
 (1)$\implies$(2): Suppose that $\Phi$ is not topologically free. 
 Let $t \in \mathbb{F} \setminus \{\emptyset\}$ be such that $\operatorname{Fix}(t)$ has non-empty interior. 
 Then, there are $\af \in \CL^*$ and $A \in \CI_\af$ such that $\emptyset \neq \CN(\af,A) \subseteq \operatorname{Fix}(t)$.
 Let $\mu=(e^{\mu_i}_{\xi_i}) \in \CN(\af,A) \subseteq \operatorname{Fix}(t)$.  Since $\mu \in U_{t}$ and $\CP(\mu)_{1,|\af|}=\af$, we have $t=\zeta\eta^{-1}$ for some $\zeta, \eta \in \CW^*$, where 
 $\zeta=\af\zeta'$ for some $\zeta' \in \CW^*$.  
 Now, since $\varphi_t(\mu)=\mu$,  there exist $\gm \in \CL^*$ such that either 
 $\CP(\mu)=\zeta\gm^\infty$, $\eta=\zeta\gm$ and $t=\zeta\gm^{-1}\zeta^{-1}$ or $\CP(\mu)=\eta\gm^\infty$, $\zeta=\eta\gm$ and $t=\eta\gm\eta^{-1}$.
 On the other hand, $\varphi_t(\mu)=\mu$ also means that  $(t, \mu) \in \operatorname{Iso} (\mathbb{F} \ltimes_\varphi \partial E)$. So, we can conclude that 
  there exist $\dt, \gm \in \CL^*$ such that  $t = \dt\gm\dt^{-1}$ or $t=\dt \gm^{-1} \dt^{-1}$, $\CP(\mu)=\dt \gm^\infty$ and  $e^{\gm_1}_{\xi_{|\dt|+1}} \cdots e^{\gm_{|\gm|}}_{\xi_{|\dt|+|\gm|}}$ is a loop at $d(e^{\dt_{|\dt|}}_{\xi_{|\dt|}})$ so that $\mu=e^{\dt_1}_{\xi_1} \cdots e^{\dt_{|\dt|}}_{\xi_{|\dt|}}(e^{\gm_1}_{\xi_{|\dt|+1}} \cdots e^{\gm_{|\gm|}}_{\xi_{|\dt|+|\gm|}})^{\infty} $.
 Here, we may assume without loss of generality that $\dt=\af\dt'$ for some $\dt' \in \CL^*$ since $t=\dt\gm^n\gm^{\pm}(\dt\gm^n)^{-1}$
 for any $n\in \N$.   
 Put $B:=\theta_{\dt'\gm}(A) \in \CI_{\dt\gm}$. 
   Then by the same arguments used in the proof of Proposition \ref{char:abelian core}, one can see that $(\gm, B)$ is a cycle with no exits. 
 Thus, $(\CB, \CL,\theta)$ dose not satisfy Condition (L).
 
 (2)$\implies$(3): Let  $U$ be  an open set in $ \partial E$. Since $\Phi$ is topologically free, the set $\{\mu \in \partial E: \mu \in U_{t^{-1}},\ \varphi_t(\mu) \neq \mu\}$ is  open dense in $\partial E$ for each $t \neq \emptyset$. Then by Baire category theorem, we have $\cap_{t \in \mathbb{F} \setminus \{\emptyset\}} \{\mu \in \partial E: \mu \in U_{t^{-1}},\ \varphi_t(\mu) \neq \mu\} $ is dense in $\partial E$. Thus, there exist $\nu \in U$ such that $\varphi_t(\nu) \neq \nu$ for all $t \neq \emptyset$. So, $(\mathbb{F} \ltimes_\varphi \partial E)_\nu^\nu=\{\nu\}$. Hence, $\mathbb{F} \ltimes_\varphi \partial E$ is topologically principal.
 
 (3)$\implies$(1): Suppose that $\mathbb{F} \ltimes_\varphi \partial E$ is topologically principal and let $(\af,A)$ be a cycle. Consider the open set $\CN(\af,A)$ in $\partial E$ and take $\mu=(e^{\mu_i}_{\eta_i}) \in \CN(\af,A)$ with trivial isotropy. Since $A \in \eta_{|\af|}$, we see that $(\af,\eta)$ is an ultrafilter cycle, and hence, 
$\mu_{[1, |\af|]}=e^{\af_1}_{\eta_1} \cdots e^{\af_{|\af|}}_{\eta_{|\af|}}$ is a loop at $h_{[\af_{|\af|}]\emptyset}(\eta_{|\af|})$.  If $\mu=(e^{\af_1}_{\eta_1} \cdots e^{\af_{|\af|}}_{\eta_{|\af|}})^\infty$, then 
$(\af^n, \mu ) \in (\mathbb{F} \ltimes_\varphi \partial E)^\mu_\mu $ for all $n \in \N$. However, this is not the case since $\mu$ has trivial isotropy. Thus, 
$\mu \neq(e^{\af_1}_{\eta_1} \cdots e^{\af_{|\af|}}_{\eta_{|\af|}})^\infty$. We then consider the following 2 cases:

If $\CP(\mu)=\af^n\af_{1,k}\gm$ for some $n \in \N$, $0 \leq k < |\af|$ and $\gm \in \CL^{\leq \infty} \setminus \{\emptyset\}$ such that $\gm_1 \neq \af_{k+1}$, then 
$\theta_{\af_{1,k}\gm_1}(A) \in \eta_{|\af^n|+k+1}$, so, $ \emptyset \neq \theta_{\af_{1,k}\gm_1}(A)$, and hence, $\Delta_{\theta_{\af_{1,k}}(A)} \neq \{\af_{k+1}\}$. Thus, $(\af,A) $ has an exit. 

If $\CP(\mu)=\af^n\af_{1,k}$ for some $n \in \N$ and $0 \leq k < |\af|$, then $d(\mu) \in E^0_{sg}$. So, $\theta_{\af_{1,k}}(A) \notin \CB_{reg}$. Thus, we have either $\Delta_{\theta_{\af_{1,k}}(A)}$ is an infinite set or there exists $\emptyset \neq B  \subseteq \theta_{\af_{1,k}}(A)$ such that $\Delta_B =\emptyset \neq \{\af_{k+1}\}$. In both cases, $(\af,A)$ has an exit.

Therefore, we can conclude that every cycle has an exit. 

(3)$\implies$(4): If $\mathbb{F} \ltimes_\varphi \partial E$  is topologically principal, then $\mathbb{F} \ltimes_\varphi \partial E$  is effective by \cite[Proposition 3.6(i)]{Renault2008}.

To prove (4)$\implies$(3), we assume that $\CB$ and $\CL$ are countable. Then $\mathbb{F} \ltimes_\varphi \partial E$ is second countable. So, if $\mathbb{F} \ltimes_\varphi \partial E$  is  effective, then $\mathbb{F} \ltimes_\varphi \partial E$ 
is topologically principal by  \cite[Proposition 3.6(ii)]{Renault2008}.
\end{proof}

\begin{prop}\label{MASA} Let $(\CB, \CL,\theta, \CI_\af)$ be a generalized Boolean dynamical system.  If $(\CB, \CL,\theta)$ satisfies Condition $(L)$, then $D$ is a MASA in $C^*(\CB, \CL,\theta, \CI_\af)$. Moreover, If $\CB$ and $\CL$ are countable, then the converse also holds true.
\end{prop}

\begin{proof} Put $\CG:=\mathbb{F} \ltimes_\varphi \partial E$ and notice that we have an  isomorphism $\kp: C^*(\CB, \CL,\theta, \CI_\af) \to C^*(\CG) $ that maps $D$ onto $C_0(\CG^{(0)}) (\cong C_0(\partial E))$.
 If $(\CB, \CL,\theta)$ satisfies Condition $(L)$, then $\CG$ is effective by Proposition \ref{equiv:Condition (L)}. Thus, by \cite[Proposition II.4.7(ii)]{Renault}, $C_0(\CG^{(0)})$ is a MASA in $C^*(\CG)$, and hence,  $D$ is a MASA in $C^*(\CB, \CL,\theta, \CI_\af)$.

For the converse, we assume that  $\CB$ and $\CL$ are countable and that $D$ is a MASA in $C^*(\CB, \CL,\theta, \CI_\af)$.  Then, $\CG$ is effective by \cite[Proposition II.4.7(ii)]{Renault}. Thus,  $(\CB, \CL,\theta)$ satisfies Condition $(L)$ by Proposition \ref{equiv:Condition (L)}.
\end{proof}

\begin{cor}  Let $(\CB, \CL,\theta, \CI_\af)$ be a generalized Boolean dynamical system.  If $(\CB, \CL,\theta)$ satisfies Condition $(L)$, then $M$ is a MASA in   $C^*(\CB, \CL,\theta, \CI_\af)$.
\end{cor}

\begin{proof} If $(\CB, \CL,\theta)$ satisfies Condition $(L)$, then $D$ is a MASA. Therefore $D'=D$. So,  $M=D$ by Lemma \ref{prop:abelian core}, and hence, $M$ is a MASA.
\end{proof}

\end{document}